\documentclass[a4paper, 11pt]{article}

\usepackage{a4wide}
\usepackage[english]{babel}
\usepackage{amsmath}
\usepackage{psfrag}
\usepackage[latin1]{inputenc}
\usepackage[T1]{fontenc}
\usepackage{amsthm}
\usepackage{amssymb}
\usepackage{verbatim}
\usepackage{stmaryrd}

\DeclareMathOperator{\res}{Res}
\DeclareMathOperator{\ind}{Ind}

\DeclareMathOperator{\syl}{Syl}
\DeclareMathOperator{\End}{End}

\DeclareMathOperator{\Stab}{Stab}
\DeclareMathOperator{\Aut}{Aut}
\DeclareMathOperator{\Inf}{Inf}

\renewcommand{\leq}{\leqslant}
\renewcommand{\geq}{\geqslant}
\renewcommand{\unlhd}{\trianglelefteqslant}

\begin{document}
\newtheorem{defi}{Definition}[section]
\newtheorem{rem}[defi]{Remark}
\newtheorem{prop}[defi]{Proposition}
\newtheorem{ques}[defi]{Question}
\newtheorem{lemma}[defi]{Lemma}
\newtheorem{cor}[defi]{Corollary}
\newtheorem{thm}[defi]{Theorem}
\newtheorem{expl}[defi]{Example}
\newtheorem{conj}[defi]{Conjecture}
\newtheorem{claim}[defi]{Claim}
\newtheorem*{main}{Theorem}
\newtheorem{noth}[defi]{}
\renewcommand{\proofname}{\textsl{\textbf{Proof}}}
\newtheorem*{conjintro}{Conjecture}
\newtheorem*{quesintro}{Question}

\begin{center}
{\bf\Large Source Algebras of Blocks, Sources of Simple Modules,\\
\vspace*{0.5em} and a Conjecture of Feit}

\bigskip
Susanne Danz\footnote{{\bf Current address:} Department of Mathematics, 
University of Kaiserslautern, P.O. Box 3049, 
         67653 Kaiserslautern, Germany, danz@mathematik.uni-kl.de} 
and J\"urgen M\"uller 

\today 

\begin{abstract}
\noindent
We verify a finiteness conjecture of Feit on sources of simple modules
over group algebras for various classes of finite groups 
related to the symmetric groups.


\noindent
{\bf MR Subject Classification:} 20C20
\end{abstract}
\end{center}

\section{Introduction and Results}\label{intro}

\begin{noth}\normalfont  
Amongst the long-standing conjectures in modular representation theory
of finite groups is a finiteness conjecture concerning
the sources of simple modules over group algebras, due to Feit \cite{Feit1}
and first announced at the Santa Cruz Conference on Finite Groups in 1979:
  
\medskip
By Green's Theorem \cite{Green}, given a finite group $G$ 
and an algebraically closed field
$F$ of some prime characteristic $p$, one can assign to
each indecomposable $FG$-module $M$ a $G$-conjugacy class 
of $p$-subgroups of $G$, the {\it vertices of $M$}.
Given a vertex $Q$ of $M$, there is, moreover, 
an indecomposable $FQ$-module $L$ such
that $M$ is isomorphic to a direct summand of the induced module
$\ind_Q^G(L)$; such a module $L$ is called a \textit{($Q$-)source} of $M$. 
Any $Q$-source of $M$ has vertex $Q$ as well, and
is determined up to
isomorphism and conjugation with elements in $N_G(Q)$.
Vertices of simple $FG$-modules have 
a number of special features not shared by
vertices of arbitrary indecomposable
$FG$-modules; see, for instance, \cite{Erd} and \cite{Kn}.
Feit's Conjecture in turn predicts also 
a very restrictive structure of sources of simple modules, and can be 
formulated as follows:

\begin{conjintro}[Feit]
Given a finite $p$-group $Q$, there are only finitely many isomorphism 
classes of indecomposable $FQ$-modules occurring as sources of
simple $FG$-modules with 
vertex $Q$; here $G$ varies over all finite groups containing 
$Q$.
\end{conjintro}

While the conjecture remains open in this generality, weaker versions of
it are known to be true:
by work of Dade \cite{Dade}, Feit's Conjecture holds when demanding
the sources in question have dimension at most $d$, for a given
integer $d$. Furthermore, Puig \cite{Puigpsol,Puig} has shown that
Feit's Conjecture holds when allowing the group $G$ to
vary over $p$-soluble groups only, and Puig \cite{Pu} 
has also shown that Feit's Conjecture holds for the symmetric groups.
The aim of this paper now is to pursue the idea of restricting to 
suitable classes of groups further, and to prove the following:

\begin{main}\label{thm:main}
Feit's Conjecture holds when letting the group $G$ vary over 
the following 
groups only:
$$\{\mathfrak{S}_n\}_{n\geq 1},\quad
  \{\mathfrak{A}_n\}_{n\geq 1},\quad
  \{\widetilde{\mathfrak{S}}_n\}_{n\geq 1},\quad
  \{\widehat{\mathfrak{S}}_n\}_{n\geq 1},\quad
  \{\widetilde{\mathfrak{A}}_n\}_{n\geq 1},\quad
  \{\mathfrak{B}_n\}_{n\geq 2},\quad
  \{\mathfrak{D}_n\}_{n\geq 4}.$$
\end{main}

Here, $\mathfrak{S}_n$ and $\mathfrak{A}_n$ denote the
symmetric and alternating groups on $n$ letters, respectively.
Moreover, $\widetilde{\mathfrak{S}}_n$ and $\widehat{\mathfrak{S}}_n$
denote the 
double covers of the symmetric groups, and
$\widetilde{\mathfrak{A}}_n$ those of the alternating groups;
these groups are described in more detail in \ref{noth:schurcov}.
Finally, $\mathfrak{B}_n$ and $\mathfrak{D}_n$ denote the Weyl groups
of type $B_n$ and $D_n$, respectively; these groups are described 
in more detail in \ref{noth:setting}. 
\end{noth}

\begin{noth}\normalfont  
We will prove the above-mentioned main result by exploiting the 
connection between Feit's Conjecture and two other ingredients:
Puig's Conjecture regarding source algebras of blocks of group algebras, 
and a question raised by Puig (see \cite{Zh}) relating vertices
to defect groups:

\begin{conjintro}[Puig]
Given a finite $p$-group $P$, there are only finitely many isomorphism 
classes of interior $P$-algebras that are source algebras of a 
block of $FG$; here $G$ varies over all finite groups
containing $P$.
\end{conjintro}

\begin{quesintro}
Suppose that $p>2$, that $G$ is a finite group,
and that $D$ is a simple $FG$-module with vertex $Q$.
Is the order of the defect groups of the block
containing $D$ bounded in terms of the group order $|Q|$?
\end{quesintro}

The corresponding question for $p=2$ is known
to have a negative answer, and we will elaborate on this 
in Remark \ref{rem:vertexbnd} in more detail.
For $p>2$, to the authors' knowledge,
there seem to be no examples known
where Puig's Question admits a negative answer, and Zhang \cite{Zh}
has proved a reduction to quasi-simple groups.
The Reduction Theorem \ref{thm:feitred} 
now shows that Puig's Conjecture
together with a positive answer to the previous question imply 
Feit's Conjecture for $p>2$.
A proof of Puig's Conjecture alone, however, might not suffice to
prove Feit's Conjecture; 
see also the remarks following \cite[Thm. 38.6]{Thev}.

\medskip

Despite the fact that these conjectures
have been around for quite a while, and belong to folklore in
modular representation theory of finite groups, we have not been
able to find a reference where they have been stated formally.
Thus they are restated here as Conjectures \ref{conj:feit} and 
\ref{conj:puig}, respectively, in a category-theoretic language 
we are going to develop, and which will also be used to 
formulate our Reduction Theorem \ref{thm:feitred}.

\medskip
Our strategy for proving our main Theorem \ref{thm:summary} is as follows: 
by work of Kessar \cite{Kess3,Kess2,Kess}, 
Puig \cite{Pu}, and Scopes \cite{Sc}, Puig's Conjecture is 
known to be true for the
groups considered here.
We will, therefore, show that Puig's Question has an
affirmative answer when allowing the groups to vary over the 
groups in the theorem only, by determining in Theorem \ref{thm:vtxbnd}
explicit upper bounds on the respective defect group orders,
and regardless of whether $p$ is even or odd.
It should be pointed out that such bounds can also be derived from work
of Zhang \cite{Zh}, which are, however, much weaker than the ones we 
get, and are thus hardly useful when actually trying to
compute vertices of simple modules in practice.

\medskip
The strategy to prove the bounds on defect group orders
in terms of vertices, in turn, is for part of the
cases based on Kn\"orr's Theorem, see Remark \ref{rem:knoerr}, 
ensuring the existence of a self-centralizing Brauer pair for 
any subgroup of a group $G$ being a vertex of a simple $FG$-module.
This reduces Puig's Question to asking, more strongly,
whether defect group orders can even be bounded in 
terms self-centralizing Brauer pairs. This idea turns out
to be successful for the alternating groups and
the 
double covers of the symmetric and alternating groups.
In particular, in Theorem \ref{thm:self} and the subsequent
Remark \ref{rem:self} we derive a detailed picture of the 
self-centralizing Brauer pairs for the alternating groups
in characteristic $p=2$, which might be of independent interest.

\medskip
Moreover, we would like to point out that, in particular, for
the case of the alternating groups, we have been examining
various examples explicitly, where the computer algebra 
systems {\sf GAP} \cite{GAP} and {\sf MAGMA} \cite{MAGMA}
have been of great help; we will specify later on where
precisely these have been invoked.
\end{noth}

\begin{noth}\normalfont
The paper is organized as follows: 
in Section \ref{sec:vsp} we introduce our notational set-up, 
define our notions of the category of interior algebras 
and vertex-source pairs of indecomposable modules over group algebras,
and recall the notion of source algebras.
Then, in Section \ref{sec:feitpuig},
we formulate Feit's and Puig's Conjectures in 
our category-theoretic language, prove the Reduction Theorem \ref{thm:feitred},
and state Theorem \ref{thm:vtxbnd} in order to prove
the main Theorem \ref{thm:summary}.
Sections \ref{sec:alt}--\ref{sec:weyl} are then
devoted to proving Theorem \ref{thm:vtxbnd} for 
the alternating groups, the 
double covers of the symmetric
and alternating groups, and the Weyl groups appearing in our
main theorem, where in the former two cases
we pursue the idea of using self-centralizing Brauer pairs,
while for the Weyl groups appearing in our
main theorem we are content with looking at vertices directly.
Finally in Section \ref{sec:semidirect} we briefly deal
with semidirect products with abelian kernel in general.

\medskip
Throughout this article, let $p$ be a prime number, and let
$F$ be a fixed algebraically closed field of characteristic $p$.
All groups appearing will be finite and, whenever $G$ is a 
group, any $FG$-module is understood to be a finitely generated
left module. 
Hence we may assume that the groups considered here
form a small category, that is, its object class is just a set,
and similarly module categories may be assumed to be small as well.
This will, for instance, allow us to speak of the {\it set} of all finite groups.
We assume the reader to be familiar with  
modular representation theory of finite groups in general, and
the standard notation commonly used, as exposed for example
in \cite{NT} and \cite{Thev}.
For background concerning the representation theory
of the symmetric groups and their covering groups, 
we refer the reader to \cite{JK} and \cite{HH}, respectively.

\medskip
{\bf Acknowledgements.} The authors' research
was supported through a Marie Curie Intra-European Fellowship
(grant PIEF-GA-2008-219543). In particular the second-named
author is grateful for the hospitality of the University of
Oxford, where part of the paper has been written.
It is a pleasure to thank Burkhard K\"ulshammer
for various helpful discussions on the topic.
We would also like to thank the referee for pointing
us to Theorem \ref{thm:semidirect}.
\end{noth}

\section{Interior Algebras, Source Algebras, and Vertex-Source Pairs}
\label{sec:vsp}

In this section we introduce a category-theoretic language,
which will be used to restate Feit's and Puig's Conjectures 
later in this article. The language we use
is that of interior algebras, see for example \cite[Ch. 10]{Thev},
where we additionally have to allow the group acting to vary. 


\begin{noth}\label{noth:A}\normalfont
{\sl A category of interior algebras.}\; We define a category $\mathcal{A}$ 
whose objects are the triples 
$(G,\alpha,A)$, where $G$ is a finite group, 
$A$ is a finite-dimensional, associative, unitary $F$-algebra, 
and $\alpha:G\longrightarrow A^{\times}$ is a group homomorphism
into the group of multiplicative units $A^{\times}$ in $A$.
Given objects $(G,\alpha,A)$ and $(H,\beta,B)$ in $\mathcal{A}$, the
morphisms $(G,\alpha,A)\longrightarrow (H,\beta,B)$ are the pairs
$(\varphi,\Phi)$ where $\varphi:G\longrightarrow H$ is a group homomorphism
and $\Phi:A\longrightarrow B$ is a homomorphism of $F$-algebras satisfying
\begin{equation}\label{equ:Phi}
\Phi(\alpha(g)a)=\beta(\varphi(g))\Phi(a)\quad\text{and}\quad 
\Phi(a\alpha(g))=\Phi(a)\beta(\varphi(g)),
\end{equation}
for all $g\in G$ and all $a\in A$. 
We emphasize that
the algebra homomorphism $\Phi$ need not be unitary, in general.
If it additionally is, that is, if we have $\Phi(1_A)=1_B$ then 
the above compatibility condition (\ref{equ:Phi}) simplifies to
$$\Phi(\alpha(g))=\beta(\varphi(g)), \quad\text{for all }g\in G .$$

\medskip
Anyway, whenever $(\varphi,\Phi):(G,\alpha,A)\longrightarrow (H,\beta,B)$
and $(\psi,\Psi):(H,\beta,B)\longrightarrow (K,\gamma,C)$ are morphisms in 
$\mathcal{A}$, 
their composition is defined to be 
$(\psi,\Psi)\circ(\varphi,\Phi):=(\psi\circ\varphi,\Psi\circ\Phi)$,
where the compositions of the respective components are the 
usual compositions of group homomorphisms and algebra homomorphisms,
respectively.
Hence $\mathcal{A}$ is indeed a category, from now on called
the \textit{category of interior algebras};
an object $(G,\alpha,A)$ in $\mathcal{A}$ is called an 
\textit{interior $G$-algebra}, and 
$(\varphi,\Phi):(G,\alpha,A)\longrightarrow (H,\beta,B)$
is called a \textit{morphism of interior algebras}.

\medskip
We just remark that for the conjugation automorphisms 
$\kappa_a\in\Aut(A)$ and $\lambda_b\in\Aut(B)$ induced by some 
$a\in (A^{\alpha(G)})^{\times}$ and $b\in (B^{\beta(H)})^{\times}$,
respectively, we also have the morphism 
$(\varphi,\lambda_b\circ\Phi\circ\kappa_a):
 (G,\alpha,A)\longrightarrow (H,\beta,B)$. This defines
an equivalence relation on the morphisms of interior algebras
$(G,\alpha,A)\longrightarrow (H,\beta,B)$, and the equivalence class
$$ (\varphi,\widehat{\Phi}):=\{(\varphi,\lambda_b\circ\Phi\circ\kappa_a)\mid
   a\in(A^{\alpha(G)})^{\times},b\in(B^{\beta(H)})^{\times}\} $$
is called the associated \textit{exomorphism of interior algebras}.

\medskip
By the above definition, $(G,\alpha,A)$ and $(H,\beta,B)$ are isomorphic in 
$\mathcal{A}$ if and only if there exists a morphism 
$(\varphi,\Phi):(G,\alpha,A)\longrightarrow (H,\beta,B)$
such that $\varphi$ is an isomorphism of groups and 
$\Phi$ is an (automatically unitary) isomorphism of algebras.
So, in particular, if $(G,\alpha,A)$ is an interior algebra and 
$\varphi:H\longrightarrow G$ is an isomorphism of groups 
then also $(H,\alpha\circ\varphi,A)$ is an interior algebra, and 
$(H,\alpha\circ\varphi,A)$ and $(G,\alpha,A)$ are isomorphic via 
$(\varphi,\mathrm{id}_A)$.
Analogously, if $(G,\alpha,A)$ is an interior algebra and if 
$\Phi:A\longrightarrow B$ is an
isomorphism of algebras then $(G,\Phi\circ\alpha,B)$ is also
an interior algebra, and $(G,\alpha,A)$ and $(G,\Phi\circ\alpha,B)$ 
are isomorphic via $(\mathrm{id}_G,\Phi)$.
\end{noth}

\begin{noth}\label{noth:pairs}
\normalfont
{\sl An equivalence relation.}\; Let $G$ and $H$ be groups, 
let $M$ be an $FG$-module, and
let $N$ be an $FH$-module. 
Let further $\alpha:G\longrightarrow E_M^{\times}$
and $\beta:H\longrightarrow E_N^{\times}$ 
be the corresponding representations,
where $E_M:=\End_F(M)$ and $E_N:=\End_F(N)$. 
Then $(G,\alpha,E_M)$ and $(H,\beta,E_N)$ are interior algebras. 

\medskip
(a)\,
We say that the pairs $(G,M)$ and $(H,N)$ are \textit{equivalent} 
if there are a group isomorphism $\varphi:G\longrightarrow H$ 
and a vector space isomorphism $\psi:M\longrightarrow N$
such that, for all $g\in G$ and all $m\in M$, we have
$$ \psi(\alpha(g)\cdot m)=\beta(\varphi(g))\cdot\psi(m).$$
This clearly is an equivalence relation on the set of all 
such pairs.

\medskip
(b)\,
The case $G=H$ deserves particular attention:
pairs $(G,M)$ and $(G,N)$ are equivalent, via $(\varphi,\psi)$ say,
if and only if we have
$$ \beta(\varphi(g))\cdot n=\psi(\alpha(g)\cdot\psi^{-1}(n)) $$
for all $g\in G$ and $n\in N$,
that is, if and only if $M$ and $N$ are in the same 
$\Aut(G)$-orbit on the set of isomorphism classes of $FG$-modules.
Moreover, $M$ and $N$ are isomorphic as $FG$-modules if and only if $\varphi$
can be chosen to be the identity $\textrm{id}_G$.
In particular, there are at most $|\Aut(G)|$ isomorphism classes of
$FG$-modules in the equivalence class of $(G,M)$.
\end{noth}

\begin{lemma}\label{lemma:pairs}
We keep the notation of \ref{noth:pairs}. 
Then the pairs $(G,M)$ and $(H,N)$ are equivalent if and only if
the associated interior algebras $(G,\alpha,E_M)$ and $(H,\beta,E_N)$ 
are isomorphic in $\mathcal{A}$. 

\smallskip
Moreover, if $G=H$ then $M$ and $N$ are isomorphic as $FG$-modules 
if and only if 
$(G,\alpha,E_M)$ and $(G,\beta,E_N)$ are isomorphic in $\mathcal{A}$ 
via an isomorphism of the form $(\mathrm{id}_G,?)$. 
\end{lemma}

\begin{proof}
If $(G,M)$ and $(H,N)$ are equivalent via
$\varphi:G\longrightarrow H$ and $\psi:M\longrightarrow N$
then 
$$ \Psi:E_M \longrightarrow E_N,\;
   \gamma\longmapsto \psi\circ\gamma\circ\psi^{-1} $$
is an isomorphism of algebras, and we have
$\Psi(\alpha(g))=\psi\circ\alpha(g)\circ\psi^{-1}=\beta(\varphi(g))$,
for all $g\in G$,
thus the interior algebras $(G,\alpha,E_M)$ and $(H,\beta,E_N)$ 
are isomorphic in $\mathcal{A}$ via $(\varphi,\Psi)$.

\medskip
Let, conversely, $(G,\alpha,E_M)$ and $(H,\beta,E_N)$ be isomorphic 
via $(\varphi,\Psi)$, where $\varphi:G\longrightarrow H$  
is a group isomorphism and $\Psi:E_M\longrightarrow E_N$
is an isomorphism of algebras. Then, letting $i\in E_M$ be a
primitive idempotent, 
we may assume that $M=E_M i$ and, letting $j:=\Psi(i)\in E_N$,
we may similarly assume that $N=E_N j$. Moreover, 
$\Psi(E_M i)=\Psi(E_M)\Psi(i)=E_N j$ shows that 
$\psi:=\Psi|_{E_M i}: E_M i\longrightarrow E_N j$
is a vector space isomorphism, where
for all $g\in G$ and $\gamma\in E_M$ we have
$$ \psi(\alpha(g)\cdot\gamma i)
= \Psi(\alpha(g)\cdot\gamma i)
= \beta(\varphi(g))\cdot\Psi(\gamma)j
= \beta(\varphi(g))\cdot\psi(\gamma i), $$
implying that $(G,M)$ and $(H,N)$ are equivalent via $(\varphi,\psi)$.
This proves the first statement.

\medskip
The second statement can be found in \cite[L. 10.7]{Thev}.
It also follows from the above observations, by recalling that
$M$ and $N$ are isomorphic $FG$-modules if and only if the
group isomorphism $\varphi:G\longrightarrow G$ inducing an
equivalence of pairs can be chosen 
to be the identity $\mathrm{id}_G$.
\end{proof}

\begin{lemma}\label{lemma:dirsummand}
We keep the notation of \ref{noth:pairs},
and let $(G,M)$ and $(H,N)$ be equivalent.
Then the equivalence classes of pairs $(G,M')$
where $M'$ is an indecomposable direct summand
of the $FG$-module $M$ coincide with the equivalence classes
of pairs $(H,N')$ where $N'$ is an indecomposable direct summand
of the $FH$-module $N$.
In particular, the $FG$-module $M$ is indecomposable 
if and only if the $FH$-module $N$ is.
\end{lemma}

\begin{proof}
Let $(G,\alpha,E_M)$ and $(H,\beta,E_N)$ be isomorphic via $(\varphi,\Psi)$;
such an isomorphism exists, by Lemma \ref{lemma:pairs}.
Given an indecomposable direct summand $M'$ of $M$, 
let $i\in (E_M)^{\alpha(G)}$ 
be the associated primitive idempotent, so that $M'=iM$,
with associated representation 
$$ \alpha': G\longrightarrow E_{iM}=iE_Mi,\; g\longmapsto i\alpha(g)i .$$
Hence, for $j:=\Psi(i)\in E_N$ we have 
$$ \beta(\varphi(g))j=\beta(\varphi(g))\Psi(i)=\Psi(\alpha(g)i)
 =\Psi(i\alpha(g))=\Psi(i)\beta(\varphi(g))=j\beta(\varphi(g)) ,$$
for all $g\in G$. Thus $j\in (E_N)^{\beta(H)}$ is a primitive idempotent, 
giving rise to the indecomposable direct summand $jN$ of $N$,
with associated representation
$$ \beta': H\longrightarrow E_{jN}=jE_Nj,\; g\longmapsto j\beta(g)j .$$
Moreover, we have an isomorphism of algebras
$$ \Psi':=\Psi|_{iE_Mi}:iE_Mi\longrightarrow jE_Nj,\;
   ixi\longmapsto\Psi(ixi)=j\Psi(x)j .$$
Then we have
$$\Psi'(\alpha'(g))=\Psi'(i\alpha(g)i)=j\Psi(\alpha(g))j
=j\beta(\varphi(g))j=\beta'(\varphi(g)) ,$$
for all $g\in G$.
Thus the interior algebras $(G,\alpha',iE_Mi)$ and $(H,\beta',jE_Nj)$ 
are isomorphic in $\mathcal{A}$ via $(\varphi,\Psi')$,
that is, the pairs $(G,iM)$ and $(H,jN)$ are equivalent.
\end{proof}

\begin{rem}\label{rem:resalg}
\normalfont
(a)
Let $(G,\alpha,A)$ be an interior algebra.
For any $A$-module $M$ with associated representation
$\delta:A\longrightarrow E_M:=\End_F(M)$ we obtain an 
$FG$-module $\res_{\alpha}(M)$ by restriction along $\alpha$, 
that is, the associated representation is given as
$\delta\circ\alpha: G\longrightarrow E_M^\times$.
Thus we get a functor
$$ \res_{\alpha}:A\textbf{-mod} \longrightarrow FG\textbf{-mod},\;
   M \longmapsto \res_{\alpha}(M) .$$

\medskip
Let $(H,\beta,B)$ be an interior algebra, and let 
$(\varphi,\Phi):(G,\alpha,A)\longrightarrow (H,\beta,B)$
be a morphism in $\mathcal{A}$.
Hence, by restriction along $\beta$ and $\varphi$, respectively,
we similarly get functors
$$ \res_{\beta}:B\textbf{-mod} \longrightarrow FH\textbf{-mod}
\quad\text{and}\quad 
\res_{\varphi}:FH\textbf{-mod} \longrightarrow FG\textbf{-mod} .$$
Moreover, for any $B$-module $N$ with associated representation
$\gamma:B\longrightarrow E_N:=\End_F(N)$ we obtain an
$A$-module $\res_{\Phi}(N):=\Phi(1_A)N$ whose
associated representation is given as
$$ A\longrightarrow\End_F(\Phi(1_A)N)=\Phi(1_A)E_N\Phi(1_A),\;
x\longmapsto \Phi(1_A)\gamma(\Phi(x))\Phi(1_A).$$
This gives rise to a functor
$\res_{\Phi}:B\textbf{-mod} \longrightarrow A\textbf{-mod}$.

\medskip
(b)
If additionally $\Phi$ is unitary, that is, $\Phi(1_A)=1_B$
then the representation associated with $\res_{\Phi}(N)$
is obtained by restriction along $\Phi$. Moreover,
from $\Phi(\alpha(g))=\beta(\varphi(g))$,
for all $g\in G$, we infer that we have the following
equality of functors
$$ \res_{\alpha}\circ\res_{\Phi} = \res_{\varphi}\circ\res_{\beta}:
B\textbf{-mod} \longrightarrow FG\textbf{-mod} .$$
In other words, for any $B$-module $N$ with associated representation
$\gamma:B\longrightarrow E_N$, we have
$$ (\gamma\circ\Phi\circ\alpha)(g)\cdot n
=(\gamma\circ\beta\circ\varphi)(g)\cdot n ,$$
for all $g\in G$ and all $n\in N$.
In particular, if $\varphi$ is an isomorphism then 
$(G,\res_{\alpha}(\res_{\Phi}(N)))$ and $(H,\res_{\beta}(N))$ 
are equivalent via $(\varphi,\textrm{id}_N)$. 
\end{rem}

\begin{defi}\label{defi:vxsce}
\normalfont
Let $G$ be a group, and let $M$ be an indecomposable $FG$-module. 
Assume that $V\leq G$ is a vertex of $M$ and that $S$ is a $V$-source of $M$.
Then the elements of the equivalence class of the pair $(V,S)$ are
called the \textit{vertex-source pairs} of $(G,M)$.
\end{defi}

\begin{prop}\label{prop:vxsce}
If $G$ is a group and $M$ is an indecomposable $FG$-module
then the vertex-source pairs of $(G,M)$ are pairwise equivalent.

\smallskip
Moreover, if $H$ is a group and $N$ is an indecomposable
$FH$-module such that $(G,M)$ is equivalent to $(H,N)$
then the vertex-source pairs of $(G,M)$ and $(H,N)$ are pairwise equivalent.
\end{prop}

\begin{proof}
Let $V\leq G$ be a vertex, and let $S$ be a $V$-source of $M$.
Then the set of all vertices of $M$ is given as $\{{}^gV\mid g\in G\}$
and, for a given $g\in G$, the set of ${}^gV$-sources 
of $M$ (up to isomorphism) is 
$\{{}^{hg}S\mid h\in N_G({}^gV)\}$. For $g\in G$ and $h\in N_G({}^gV)$, let 
$$ \kappa: V\longrightarrow {}^gV={}^{hg}V,\;
   x\longmapsto {}^{hg}x=hgxg^{-1}h^{-1} $$
be the associated conjugation homomorphism, and let 
$\psi: S\longrightarrow {}^{hg}S,\; m\longmapsto hg\otimes m$.
Then for all $x\in V$ and $m\in S$, we have 
$$ \kappa(x)\cdot\psi(m)=({}^{hg}x)\cdot(hg\otimes m)
  =hg\otimes (x\cdot m)=\psi(x\cdot m) ,$$
hence the pairs $(V,S)$ and $({}^gV,{}^{hg}S)$ are equivalent
via $(\kappa,\psi)$. 
Since every vertex-source pair of $(G,M)$ is equivalent
to one of the pairs $({}^gV,{}^{hg}S)$, 
this shows that all vertex-source pairs
of $(G,M)$ belong to the same equivalence class.

\medskip
Moreover, if $(G,M)$ and $(H,N)$ are equivalent via $(\varphi,\psi)$ then
$\psi(\alpha(g)\cdot m)=\beta(\varphi(g))\cdot\psi(m)$
for all $g\in G$ and $m\in M$, where $\alpha$ and $\beta$ are
the representations associated with $M$ and $N$, respectively. From this
we infer that $\varphi(V)$ is a vertex of the $FH$-module $N$ having
$\psi(S)$ as a $\varphi(V)$-source, and that
$(V,S)$ is equivalent to $(\varphi(V),\psi(S))$ via $(\varphi|_V,\psi|_S)$.
\end{proof}

\begin{rem}\label{rem:non-vertex-source-pairs}
\normalfont
We remark that, given $G$ and an indecomposable $FG$-module $M$,
specifying a vertex $V$ as a subgroup of $G$ amounts to 
restricting to those vertex-source pairs of shape $(V,?)$, henceforth
only allowing for isomorphisms of the form $(\mathrm{id}_V,?)$.
The above argument now shows that these vertex-source pairs
are given by the $FV$-modules $\{{}^h S\mid h\in N_G(V)/VC_G(V)\}$,
where $S$ is one of the $V$-sources. Thus we 
possibly do not obtain the full $\Aut(V)$-orbit of $S$,
but only see its orbit under $N_G(V)/VC_G(V)\leq\Aut(V)$,
as the following example shows:
\end{rem}

\begin{expl}\label{expl:non-vertex-source-pairs}
\normalfont
Let $p:=2$, let $G:=\mathfrak{S}_6$,
and let $M:=D^{(5,1)}$ be the natural simple $F\mathfrak{S}_6$-module
of $F$-dimension $4$. Then, by \cite{MueZim}, the vertices of $D^{(5,1)}$
are the Sylow $2$-subgroups of $\mathfrak{S}_6$.
Let $P_6:=P_4\times P_2\cong D_8\times C_2$, where $P_4=\langle (1,2),(1,3)(2,4)\rangle$ 
and $P_2:=\langle (5,6)\rangle$. Then $P_6$ is a Sylow $2$-subgroup
of $\mathfrak{S}_6$
and, by \cite{MueZim}, the restriction
$S:=\res_{P_6}^{\mathfrak{S}_6}(D^{(5,1)})$ is indecomposable, thus
every $P_6$-source of $D^{(5,1)}$ is isomorphic to $S$. 
Since $N_{\mathfrak{S}_6}(P_6)=P_6$, 
in view of Proposition \ref{prop:vxsce} we have to show that
the $\Aut(P_6)$-orbit of $S$ consists of more than a single
isomorphism class of $FP_6$-modules.

\medskip
Let $\varphi\in\Aut(P_6)$ be the involutory automorphism given 
by fixing $P_4=\langle (1,2),(1,3)(2,4)\rangle$ and mapping
$(5,6)$ to $(1,2)(3,4)(5,6)$. 
Since $\res_{\mathfrak{S}_4}^{\mathfrak{S}_6}(D^{(5,1)})$
is the natural permutation $F\mathfrak{S}_4$-module,
there is an $F$-basis of $S$ with respect to which the 
elements of $P_4$ are mapped to the associated permutation matrices, while 
$$
(5,6)\longmapsto 
\begin{bmatrix}
.&1&1&1 \\
1&.&1&1 \\
1&1&.&1 \\
1&1&1&. \\
\end{bmatrix}\quad\text{and}\quad
\varphi((5,6))=(1,2)(3,4)(5,6)\longmapsto 
\begin{bmatrix}
1&1&.&1 \\
1&1&1&. \\
.&1&1&1 \\
1&.&1&1 \\
\end{bmatrix}.
$$
It can be checked, for example with the help of the  
computer algebra system {\sf MAGMA} \cite{MAGMA},
that the $FP_6$-modules $S$ and ${}^\varphi S$ are not isomorphic.
\end{expl}


\begin{noth}\label{noth:scealg}
\normalfont
{\sl Source algebras.}\;
Let $G$ be a group, and let $B$ be a block of $FG$. 
Let further $P$ be a $p$-group
such that the defect groups of $B$ are isomorphic to $P$. 
Then we have an embedding of groups
$f:P\longrightarrow G$
such that $f(P)$ is a defect group of $B$.
The block $B$ is an indecomposable $F[G\times G]$-module 
with vertex $\Delta f(P)$ and trivial source. 
Moreover, there is an indecomposable direct summand $M$ of
$\res_{f(P)\times G}^{G\times G}(B)$ with vertex $\Delta f(P)$, 
where $M$ is unique up to isomorphism and conjugation in $N_G(f(P))$.

\medskip
So there is a 
primitive idempotent $i\in B^{f(P)}$ such that $M=iB$,
where $i$ is unique up to taking associates in $B^{f(P)}$
and conjugates under the action of $N_G(f(P))$.
We call $i$ and $M$, respectively, 
a \textit{source idempotent} and a \textit{source module} of
$B$, respectively; as a general reference see \cite[Ch. 38]{Thev}.
The embedding $f$ gives rise to the group homomorphism
$$\alpha_{f,i}:P\longrightarrow (iBi)^{\times},\; g\longmapsto if(g)i ,$$
which turns $(P,\alpha_{f,i},iBi)$ into an interior $P$-algebra.
Note that $\alpha_{f,i}$ is injective, by \cite[Exc. 38.2]{Thev}.
We call an interior algebra that is isomorphic to $(P,\alpha_{f,i},iBi)$
in $\mathcal{A}$ a \textit{source algebra} of $B$.
\end{noth}

\begin{prop}\label{prop:scealg}
Let $G$ be a group, and let $B$ be a block of $FG$. 
Then the source algebras of
$B$ are pairwise isomorphic in $\mathcal{A}$. 

\smallskip
Moreover, if $\psi:G\longrightarrow G'$ is a group isomorphism
and if $B':=\psi(B)$ is the block of $FG'$ obtained by
extending $\psi$ to $FG$ then the source algebras of 
$B$ and $B'$ are pairwise isomorphic in $\mathcal{A}$.
\end{prop}

\begin{proof}
Let $P$ be a $p$-group isomorphic to the defect groups of $B$,
let $f:P\longrightarrow G$ be an embedding such that $f(P)\leq G$ is
a defect group of $B$, and let 
$\alpha_{f,i}:P\longrightarrow (iBi)^{\times},\; g\longmapsto if(g)i$
be the associated group 
homomorphism, where $i\in B^{f(P)}$ is a source idempotent.
Moreover, let $\varphi:Q\longrightarrow P$ be a group isomorphism, 
and let $f':Q\longrightarrow G$ be an embedding 
such that $f'(Q)\leq G$ is a defect group of $B$, with 
associated group homomorphism $\alpha_{f',j}:Q\longrightarrow (jBj)^{\times}$,
where $j\in B^{f'(Q)}$ is a source idempotent. 
Note that, hence, there is some $h\in G$ such that $f'(Q)={}^h f(P)$,
and the idempotents ${}^h i=hih^{-1}$ and $j$ are associate in $B^{f'(Q)}$.
To show that $(P,\alpha_{f,i},iBi)$ is isomorphic to $(Q,\alpha_{f',j},jBj)$ 
in $\mathcal{A}$ we proceed in three steps:

\medskip
(i) 
We first consider the particular case where $Q=P$, 
$\varphi=\textrm{id}$, and $f'=f$, and let $j\in B^{f(P)}$ 
be a source idempotent that is associate to $i$.
Then there is some $a\in (B^{f(P)})^{\times}$ such that $j={}^ai=aia^{-1}$, and 
$\kappa: iBi\longrightarrow jBj,\; x=ixi\longmapsto {}^a(ixi)=j({}^ax)j$
is an isomorphism of algebras.
Hence, we obtain the group homomorphism
$$ \kappa\circ\alpha_{f,i}: 
   P\longrightarrow (jBj)^{\times},\;
   g\longmapsto {}^a(if(g)i)=j(af(g)a^{-1})j=jf(g)j ,$$
that is, $\kappa\circ\alpha_{f,i}=\alpha_{f',j}$.
Moreover, 
$\kappa(\alpha_{f,i}(g))
={}^a(if(g)i)
=j{}f(g)j
=\alpha_{f',j}(g)$,
for all $g\in P$, 
shows that the interior $P$-algebras $(P,\alpha_{f,i},iBi)$ and 
$(P,\alpha_{f',j},jBj)$ are isomorphic via $(\textrm{id}_P,\kappa)$.

\medskip
(ii)
Next, let still $\varphi=\textrm{id}$, 
let $h\in G$ be arbitrary with associated conjugation automorphism
$G\longrightarrow G,\; g\longmapsto {}^h g=hgh^{-1}$,
and let $f':P\longrightarrow G,\; g\longmapsto hf(g)h^{-1}$ 
be the associated conjugated embedding. 
Since $j\in B^{f'(P)}$ is a source idempotent, by (i)
we may assume that $j={}^h i$.
This yields the isomorphism of algebras
$\gamma: iBi\longrightarrow jBj,\; x=ixi\longmapsto {}^h(ixi)=j({}^hx)j$
and, associated to $f'$, the group homomorphism
$$ \alpha_{f',j}=\gamma\circ\alpha_{f,i}: 
   P\longrightarrow (jBj)^{\times},\; 
   g\longmapsto {}^h(if(g)i)=j(hf(g)h^{-1})j .$$
Moreover, 
$\gamma(\alpha_{f,i}(g))
={}^h(if(g)i)
=j(hf(g)h^{-1})j
=\alpha_{f',j}(g)$,
for all $g\in P$,
shows that the interior $P$-algebras $(P,\alpha_{f,i},iBi)$ and 
$(P,\alpha_{f',j},jBj)$ are isomorphic via $(\textrm{id}_P,\gamma)$.

\medskip
(iii)
We finally consider the general case 
of a group isomorphism $\varphi:Q\longrightarrow P$
and an embedding $f':Q\longrightarrow G$ as above.
By (ii) we may assume that $f'(Q)=f(P)$. Hence, 
there is a group automorphism $\rho: f(P)\longrightarrow f(P)$
such that $f\circ\varphi=\rho\circ f'$. Thus,
replacing $\varphi$ by $\varphi':=(f^{-1}\circ\rho^{-1}\circ f)\circ\varphi$
we get $f\circ\varphi'=f'$. So we may assume that $f\circ\varphi=f'$.
Moreover, by (ii) we may assume that $j=i\in B^{f(P)}$.
Hence we have the associated group homomorphisms
$$ \alpha_{f,i}: P\longrightarrow (iBi)^{\times},\; g\longmapsto if(g)i
\quad\text{and}\quad
\alpha_{f',i}: Q\longrightarrow (iBi)^{\times},\;
       h\longmapsto if'(h)i=if(\varphi(h))i .$$
Moreover,
$\alpha_{f',i}(h)
=if(\varphi(h))i
=\alpha_{f,i}(\varphi(h))$,
for all $h\in Q$,
which shows that the interior algebras 
$(Q,\alpha_{f',i},iBi)$ and $(P,\alpha_{f,i},iBi)$  
are isomorphic via $(\varphi,\textrm{id}_{iBi})$.
This proves the first statement.

\medskip
Let $\psi:G\longrightarrow G'$ be a group isomorphism, which extends
to an $F$-algebra isomorphism $\psi:FG\longrightarrow FG'$,
and let $B':=\psi(B)$.
Then, letting $f':=\psi\circ f: P\longrightarrow G'$, we conclude that
$f'(P)$ is a defect group of $B'$, and $i':=\psi(i)\in (B')^{f'(P)}$
is a source idempotent of $B'$.
Thus we get $(P,\alpha_{f',i'},i'B'i')$ as a source algebra associated
with $B'$, where 
$$ \alpha_{f',i'}=\psi\circ\alpha_{f,i}:
   P\longrightarrow (i'B'i')^{\times}=\psi(iBi)^{\times},\;
   g\longmapsto i'f'(g)i'=\psi(if(g)i) .$$
Then we have 
$\psi(\alpha_{f,i}(g))=\alpha_{f',i'}(g)$,
for all $g\in P$,
that is,
$(P,\alpha_{f,i},iBi)$ and $(P,\alpha_{f',i'},i'B'i')$ are isomorphic in 
$\mathcal{A}$ via $(\textrm{id}_P,\psi|_{iBi})$.
Since, by what we have shown above, every source algebra of $B$ is
$\mathcal{A}$-isomorphic to $(P,\alpha_{f,i},iBi)$
and every source algebra of $B'$ is $\mathcal{A}$-isomorphic to
$(P,\alpha_{f',i'},i'B'i')$, this completes the proof of the proposition.
\end{proof}

\begin{rem}\label{rem:non-source-algebras}
\normalfont
We remark that specifying a defect group $P$ of the
block $B$ as a subgroup of $G$ amounts to keeping the embedding
$f:P\longrightarrow G$ fixed, and thus to
restricting to the source algebras of shape $(P,\alpha_{f,i},iBi)$,
for some $i\in B^{f(P)}$, and to isomorphisms of the form 
$(\mathrm{id}_P,\Phi)$. 
The above argument now shows that the isomorphisms
$\Phi$ realized in $G$ are precisely those of the form 
$\Phi=\gamma\circ\kappa$,
where $\kappa:iBi\longrightarrow iBi,\; x\longmapsto {}^a x$
is the inner automorphism of $iBi$ induced by some $a\in (iBi)^\times$,
and where $\gamma:iBi\longrightarrow jBj,\; x\longmapsto {}^h x$
is induced by the conjugation automorphism $G\longrightarrow G$
afforded by some $h\in N_G(f(P))$, where $j:={}^h i$.
Hence possibly not all elements
of the isomorphism class of $(P,\alpha_{f,i},iBi)$ are realized 
as source algebras in this strict sense,
as the following example shows:
\end{rem}

\begin{expl}\label{expl:non-source-algebras}
\normalfont
Let $p:=3$ and let $G=P=\langle z\rangle\cong C_3$ 
be the cyclic group of order $3$; hence $FP$ is a local $F$-algebra.
Letting $f=\textrm{id}_P:P\longrightarrow P$,
\textit{the} source algebra of $B=FP$ (by necessarily taking $i:=1_{FP}$)
is given as $(P,\alpha_{\textrm{id}_P},FP)$.
Thus $(P,\alpha_{\textrm{id}_P},FP)$ is the only interior algebra
in its isomorphism class that is actually realized in the 
above strict sense.

\medskip
We describe all interior algebras $(P,?,FP)$ isomorphic to
$(P,\alpha_{\textrm{id}_P},FP)$ in $\mathcal{A}$,
that is, all source algebras of $FP$ in the sense of 
\ref{noth:scealg}:
note first that in this particular case any group automorphism of $P$ 
can be extended uniquely to an algebra automorphism of $FP$, so that 
any isomorphism $(P,\alpha_{\textrm{id}_P},FP)\longrightarrow (P,?,FP)$ 
is of the form $(\mathrm{id}_P,\Phi)$, where $\Phi\in\Aut(FP)$ is
an algebra automorphism of $FP$.

\medskip
Letting $y:=1-z\in FP$, hence $y^2=1+z+z^2$,
the $F$-basis $\{1,y,y^2\}$ is adjusted to the radical series
$FP=J^0(FP)>J^1(FP)>J^2(FP)>J^3(FP)=\{0\}$ of $FP$, and it
can be checked, for example with the help of the
computer algebra system {\sf GAP} \cite{GAP},
that, with respect to this basis, we have 
$$ \Aut(FP)\cong\left\{\Phi_{a,b}:=\begin{bmatrix}
1&.&. \\ .&a&. \\ .&b&a^2 \\ \end{bmatrix}\in\text{GL}_3(F)\, \Big\arrowvert\,
a\in F^\times,b\in F \right\}. $$
Hence we have $\Phi_{a,b}(z)=(1-a-b)+(a-b)z-bz^2$;
in particular, we have $\textrm{id}_{FP}=\Phi_{1,0}$, 
and the non-trivial automorphism of $P$, mapping
$z\longmapsto z^2$, extends to $\Phi_{-1,-1}$.
Thus the interior algebras looked for are given as
$(P,\Phi_{a,b}\circ\alpha_{\textrm{id}_P},FP)$, where 
$(P,\alpha_{\textrm{id}_P},FP)=(P,\Phi_{1,0}\circ\alpha_{\textrm{id}_P},FP)$.

\medskip
Finally note that this does not encompass all possible embeddings
$P\longrightarrow (FP)^\times$:
since $-z-z^2\in (FP)^\times$ has order $3$, 
there is an embedding of groups
$\beta:P\longrightarrow (FP)^\times,\; z\longmapsto -z-z^2$,
which extends to the unitary algebra endomorphism $\Phi_{0,1}$ of $FP$,
which is not an automorphism. Anyway, this gives rise to the interior 
algebra $(P,\beta,FP)$,
which is not isomorphic to $(P,\alpha_{\textrm{id}_P},FP)$ in $\mathcal{A}$,
hence is not a source algebra of $FP$.
\end{expl}

\section{Reducing Feit's Conjecture to Puig's Conjecture}
\label{sec:feitpuig}

We have now prepared the language to state Feit's Conjecture 
on sources of simple modules over group algebras as well as
Puig's Conjecture on source algebras of blocks precisely.
We will then prove the reduction theorem relating these
conjectures, which we will use extensively throughout this paper.

\begin{noth}\label{noth:scealgvx}
\normalfont
{\sl Source algebras vs vertex-source pairs.}\; 
(a) The relation between source algebras,
in the sense of \ref{noth:scealg}, and vertex-source pairs, 
in the sense of Definition \ref{defi:vxsce}, is given as follows:
let $G$ be a group, let $B$ be a block of $FG$,
let $f:P\longrightarrow G$ be an embedding such that 
$f(P)\leq G$ is a defect group of $B$, and let 
$(P,\alpha_{f,i},iBi)$ be a source algebra of $B$.
Then, by \cite[Prop. 38.2]{Thev},
we have a Morita equivalence between the algebras $B$ and $iBi$, 
in the language of Remark \ref{rem:resalg}
given by the restriction functor 
$$ \res_\Psi:B\textbf{-mod} \longrightarrow iBi\textbf{-mod} $$
with respect to the natural embedding of algebras 
$\Psi:iBi\longrightarrow B$.

\medskip
Suppose that $M$ is an indecomposable 
$FG$-module belonging to the block $B$. 
Then the Morita correspondent of $M$ in $iBi$ is 
$\res_\Psi(M)=iM$. 
Moreover, restricting $iM$ along $\alpha_{f,i}$, we get an $FP$-module 
$\res_{\alpha_{f,i}}(iM)$, which is, in general, decomposable. 
By \cite[Prop. 38.3]{Thev}, the vertex-source pairs of $(G,M)$ 
are precisely the vertex-source pairs $(Q,?)$ of the 
indecomposable direct summands of the $FP$-module $\res_{\alpha_{f,i}}(iM)$
such that $|Q|$ is maximal.

\medskip
(b)
We show that proceeding like this to determine
the vertex-source pairs of $(G,M)$
is independent of the particular choice of a source algebra:
let $(D,\alpha,A)$ be any source algebra of $B$. Hence,
by Proposition \ref{prop:scealg}, there is an isomorphism
$(\varphi,\Phi):(D,\alpha,A)\longrightarrow(P,\alpha_{f,i},iBi)$
in $\mathcal{A}$. By Remark \ref{rem:resalg}, we have
an equivalence
$$ \res_{\Phi}:iBi\textbf{-mod} \longrightarrow A\textbf{-mod} .$$

\medskip
Letting $N:=\res_{\Phi}(iM)$,
we infer that the pairs $(P,\res_{\alpha_{f,i}}(iM))$ and
$(D,\res_{\alpha}(N))$ are equivalent.
Hence, by Lemma \ref{lemma:dirsummand},
the equivalence classes of pairs $(P,M')$ where 
$M'$ is a direct summand of the $FP$-module $iM$
coincide with the equivalence classes of pairs $(D,N')$
where $N'$ is a direct summand of the $FD$-module $N$.
Moreover, if the pairs $(P,M')$ and $(D,N')$ are equivalent
then, by Proposition \ref{prop:vxsce}, their vertex-source
pairs are pairwise equivalent.

\medskip
In conclusion, to find the vertex-source pairs of $(G,M)$,
we may go over from $(P,\alpha_{f,i},iBi)$ to an arbitrary source algebra 
$(D,\alpha,A)$ by considering the module $N$ instead,
and check the above maximality condition by varying over
the pairs $(D,N')$.
\end{noth}

\begin{defi}\label{defi:vg}
\normalfont
Let $\mathcal{G}$ be a set 
of groups, and let $Q$ be a $p$-group.

\medskip
(a)
We define $\mathcal{V}_{\mathcal{G}}(Q)$ to be
the set of all equivalence classes of pairs $(Q,L)$
where $L$ is an indecomposable $FQ$-module
such that $(Q,L)$ is a vertex-source pair of some
pair $(G,M)$, where $G$ is a group in $\mathcal{G}$ 
and $M$ is a \textit{simple} $FG$-module.
In the case that $\mathcal{G}$ is the set 
of all (finite) groups,
we also write $\mathcal{V}(Q)$ rather than $\mathcal{V}_{\mathcal{G}}(Q)$.

\medskip
(b)
We say that $\mathcal{G}$ has the
\textit{vertex-bounded-defect property} with respect to $Q$ if
there is an integer $c_{\mathcal{G}}(Q)$ such that,
for every pair $(Q,L)$ in $\mathcal{V}_{\mathcal{G}}(Q)$
and for every pair $(G,M)$
consisting of a group $G$ in $\mathcal{G}$ and a {\it simple} $FG$-module $M$
having $(Q,L)$ as a vertex-source pair, 
$M$ belongs to a block of $FG$ 
having defect groups of order at most $c_{\mathcal{G}}(Q)$.
\end{defi}

\begin{rem}\label{rem:vertexbnd}
\normalfont
The vertex-bounded-defect property, by \cite{Erd},
holds in the case where $Q$ is cyclic, with $c(Q)=|Q|$, including the case $Q=\{1\}$,
covering all blocks of finite representation type.
But it does indeed not hold in general, where, in particular,
in the realm of blocks of tame representation type
there are prominent counterexamples:

\medskip
Let $p=2$. For the groups 
$\{\textrm{PSL}_2(q)\mid q\equiv 1\pmod{4}\}$,
the Sylow $2$-subgroups are isomorphic
to the dihedral group $D_{(q-1)_2}$, where $(q-1)_2$ denotes the 
$2$-part of $q-1$. Also, there is a simple $F[\textrm{PSL}_2(q)]$-module
in the principal block having dimension $(q-1)/2$ and whose vertices,
by \cite{E}, are isomorphic to the Klein four-group $V_4\cong C_2\times C_2$.
Moreover, for the groups 
$\{\textrm{SL}_2(q)\mid q\equiv 1\pmod{4}\}$,
consisting of the 
universal covering groups of groups above, 
the Sylow $2$-subgroups are isomorphic 
to the generalized quaternion group $\mathfrak{Q}_{2(q-1)_2}$,
and the inflations of the above simple $F[\textrm{PSL}_2(q)]$-modules
to $F[\textrm{SL}_2(q)]$ have vertices isomorphic to the 
quaternion group $\mathfrak{Q}_8$.
Finally, for the groups 
$\{\textrm{GU}_2(q)\mid q\equiv 1\pmod{4}\}$
the Sylow $2$-subgroups are isomorphic to the semidihedral group 
$\textrm{SD}_{4(q-1)_2}$, and the identification 
$\textrm{SL}_2(q)\cong\textrm{SU}_2(q)$ shows that there is a simple 
$F[\textrm{GU}_2(q)]$-module in the principal block having dimension 
$q-1$ whose vertices are isomorphic to $V_4$.
(Alternatively, for the groups 
$\{\textrm{PSL}_3(q)\mid q\equiv 3\pmod{4}\}$, 
the Sylow $2$-subgroups are isomorphic to the semidihedral group 
$\textrm{SD}_{2(q+1)_2}$, and there is a simple 
$F[\textrm{PSL}_3(q)]$-module in the principal block having 
dimension $q(q+1)$ whose vertices, by \cite{EE}, are isomorphic to $V_4$.)

\medskip

\mbox{}From these cases we also obtain blocks of wild 
representation type violating the vertex-bounded-defect property,
for example by taking direct products.
Hence the question arises for which defect groups $P$ or 
groups $\mathcal{G}$ one might expect the vertex-bounded-defect 
property to hold.
In particular, the following is in \cite{Zh} attributed to Puig:
\end{rem}

\begin{ques}
If $p$ is odd, does then $\mathcal{G}$
always have the vertex-bounded-defect property with respect to $Q$?
\end{ques}

\medskip
We can now state Feit's and Puig's Conjectures, 
and prove the reduction theorem.

\begin{conj}[Feit \cite{Feit1}]\label{conj:feit}
Let $\mathcal{G}$ be a set 
of groups
(which might, in particular, be the set 
of all groups), 
let $Q$ be a $p$-group, and let $\mathcal{V}_{\mathcal{G}}(Q)$
denote the set of equivalence classes of vertex-source 
pairs introduced in Definition \ref{defi:vg}.
Then $\mathcal{V}_{\mathcal{G}}(Q)$ is finite.
\end{conj}

In consequence of Lemma \ref{lemma:pairs}, 
we can reformulate Feit's Conjecture 
equivalently also in the following way:

\begin{conj}\label{conj:feit'}
Let $\mathcal{G}$ be a set 
of groups, and let
$Q$ be a $p$-group. 
Then there are, up to isomorphism in $\mathcal{A}$, only finitely many
interior algebras $(Q,\alpha,E_L)$,
where $E_L=\End_F(L)$ for an indecomposable $FQ$-module $L$
with corresponding representation $\alpha:Q\longrightarrow E_L^{\times}$,
such that $(Q,L)$ is a vertex-source pair of some
pair $(G,M)$, where $G$ is a group in $\mathcal{G}$
and $M$ is a \textit{simple} $FG$-module.
\end{conj}

\begin{conj}[Puig]\label{conj:puig}
Let $\mathcal{G}$ be a set  
of groups
(which might, in particular, be the set 
of all groups), 
and let $P$ be a $p$-group. 
Then there are only finitely many
$\mathcal{A}$-isomorphism classes of interior algebras of
$p$-blocks of groups in $\mathcal{G}$ whose 
defect groups are isomorphic to $P$.

\medskip
\normalfont
As for the origin of this conjecture, see \cite[Conj. 38.5]{Thev},
and the comment on \cite[p. 340]{Thev}.
\end{conj}

\begin{thm}\label{thm:feitred}
Let $\mathcal{G}$ be a set 
of groups satisfying
the vertex-bounded-defect property with respect to any
$p$-group. Suppose that Puig's Conjecture \ref{conj:puig}
holds true for $\mathcal{G}$. Then Feit's Conjecture \ref{conj:feit} 
is true for $\mathcal{G}$ as well.
\end{thm}

\begin{proof}
Let $Q$ be a $p$-group, and let $c_{\mathcal{G}}(Q)$ 
be the integer appearing in Definition \ref{defi:vg}. 
Then there are finitely many
(mutually non-isomorphic) $p$-groups
$R_1,\ldots, R_n$ such that, whenever $G\in\mathcal{G}$ and 
$M$ is a simple $FG$-module with vertex isomorphic to $Q$,
the defect groups of the block containing $M$ 
are isomorphic to one of the groups in $\{R_1,\ldots,R_n\}$.

\medskip
Let $k\in\{1,\ldots,n\}$. Then, by Puig's Conjecture,
there are, up to isomorphism in $\mathcal{A}$, only finitely 
many interior $R_k$-algebras occurring as source algebras of $p$-blocks 
for groups in $\mathcal{G}$ with defect groups isomorphic to $R_k$. 
Denote by 
$\{(R_k,\alpha_{k,1},A_{k,1}),\ldots, (R_k,\alpha_{k,l_k},A_{k,l_k})\}$ 
a transversal for these isomorphism classes. 

\medskip
Let further $r\in\{1,\ldots,l_k\}$, and choose representatives
$\{M_{k,r,1},\ldots,M_{k,r,d_{k,r}}\}$ 
for the isomorphism classes of simple $A_{k,r}$-modules. 
Via restriction along $\alpha_{k,r}$ we get $FR_k$-modules 
$\res_{\alpha_{k,r}}(M_{k,r,1}),\ldots,\res_{\alpha_{k,r}}(M_{k,r,d_{k,r}})$. 
For each $i\in\{1,\ldots,d_{k,r}\}$ 
we determine a vertex-source pair $(Q_{k,r,i},S_{k,r,i})$
of an indecomposable direct summand of $\res_{\alpha_{k,r}}(M_{k,r,i})$
such that $|Q_{k,r,i}|$ is maximal.
So this gives rise to the finite set of pairs
$$ \mathcal{V}:=\bigcup_{k=1}^n\bigcup_{r=1}^{l_k}\bigcup_{i=1}^{d_{k,r}}
              \{(Q_{k,r,i},S_{k,r,i})\} .$$

\medskip
Consequently, by \cite[Prop. 38.3]{Thev}, any vertex-source pair
of some pair $(G,M)$, with $G\in\mathcal{G}$ and $M$ a simple $FG$-module,
is equivalent
to one of the pairs in the finite set $\mathcal{V}$.
Hence $\mathcal{V}_{\mathcal{G}}(Q)$ is finite, proving Feit's Conjecture.
\end{proof}


\medskip
To prove Feit's Conjecture for the 
groups listed 
in the main theorem, 
we are going to apply Theorem \ref{thm:feitred}.
In order to do so, we will show that each of these sets 
satisfies the vertex-bounded-defect property with respect to
any $p$-group; this will be done by giving explicit bounds as in the next theorem, 
whose proof will be broken up into several steps in subsequent sections.

\begin{thm}\label{thm:vtxbnd}
Let $Q$ be a $p$-group, let $G$ be a finite group possessing a simple
$FG$-module $M$ belonging to a block with defect group isomorphic to $P$,
and having vertices isomorphic to $Q$. Then the following hold:

\smallskip

\quad {\rm (a)}\, 
If $G=\mathfrak{S}_n$ then $|P|\leq |Q|!$.
\smallskip

\quad {\rm (b)}\, 
If $G=\mathfrak{A}_n$ and $p=2$ 
then $|P|\leq (|Q|+2)!/2$.
\smallskip

\quad{\rm (c)}\, 
If $G\in\{\widetilde{\mathfrak{S}}_n,\widehat{\mathfrak{S}}_n\}$ 
and $p\geq 3$ then $|P|\leq |Q|!$.
\smallskip

\quad {\rm (d)}\, 
If $G=\mathfrak{B}_n$ and $p\geq 3$ then $|P|\leq |Q|!$.
\smallskip

\quad {\rm (e)}\, 
If $G=\mathfrak{B}_n$ and $p=2$ then $|P|\leq |Q|\cdot\log_2(|Q|)!.$
\smallskip

\quad {\rm (f)}\, 
If $G=\mathfrak{D}_n$ and $p=2$ then $|P|\leq |Q|\cdot(\log_2(|Q|)+1)!.$
\end{thm}

\begin{proof}
\smallskip

\quad {\rm (a)}\, 
follows from \cite[Thm. 5.1]{DK}.

\smallskip

\quad {\rm (b)}\, 
is proved in Proposition \ref{prop:boundAn}.

\smallskip

\quad {\rm (c)}\, 
is proved in Proposition \ref{prop:boundcover}.

\smallskip

\quad {\rm (d)}\,
is proved in Proposition \ref{prop:puig_weylB}.

\smallskip

\quad {\rm (e) and (f)}\,
are proved in \ref{noth:weyleven}. 
\end{proof}

Using Theorems \ref{thm:feitred} and \ref{thm:vtxbnd}, 
we are now in a position to prove our main result:

\begin{thm}\label{thm:summary}
Feit's Conjecture holds for the following 
groups:
$$\{\mathfrak{S}_n\}_{n\geq 1},\quad 
  \{\mathfrak{A}_n\}_{n\geq 1},\quad 
  \{\widetilde{\mathfrak{S}}_n\}_{n\geq 1},\quad 
  \{\widehat{\mathfrak{S}}_n\}_{n\geq 1},\quad 
  \{\widetilde{\mathfrak{A}}_n\}_{n\geq 1},\quad 
  \{\mathfrak{B}_n\}_{n\geq 2},\quad 
  \{\mathfrak{D}_n\}_{n\geq 4}.$$ 
\end{thm}

\begin{proof}
(i) Let $\mathcal{G}=\{\mathfrak{S}_n\}_{n\geq 1}$. Then Puig's Conjecture
holds for $\mathcal{G}$, by work of Puig \cite{Pu} and Scopes \cite{Sc}.
Moreover, $\mathcal{G}$ has the vertex-bounded-defect property 
with respect to any $p$-group, by Theorem \ref{thm:vtxbnd}(a).
Hence Feit's Conjecture holds, by Theorem \ref{thm:feitred}.

\medskip
(ii) Let $\mathcal{G}=\{\mathfrak{A}_n\}_{n\geq 1}$. 
Suppose first that $p\geq 3$, 
and let $E$ be a simple $F\mathfrak{A}_n$-module with 
vertex-source pair $(Q,L)$. Then there is a simple 
$F\mathfrak{S}_n$-module $D$ such that
$E\mid\res_{\mathfrak{A}_n}^{\mathfrak{S}_n}(D)$. 
Furthermore, $(Q,L)$ is also a vertex-source pair of $D$. Hence we have 
$\mathcal{V}_{\mathcal{G}}(Q)\subseteq\mathcal{V}_{\{\mathfrak{S}_n\}}(Q)$,
and we are done using (i).

\medskip
Let now $p=2$. Then Puig's Conjecture
holds for $\mathcal{G}$, by work of Kessar \cite{Kess}.
Moreover, $\mathcal{G}$ has the vertex-bounded-defect property
with respect to any $p$-group, by Theorem \ref{thm:vtxbnd}(b).
Hence Feit's Conjecture holds, by Theorem \ref{thm:feitred}.

\medskip
(iii) Let $\mathcal{G}=\{\widetilde{\mathfrak{S}}_n\}_{n\geq 1}$,
where we may argue identically for 
$\mathcal{G}=\{\widehat{\mathfrak{S}}_n\}_{n\geq 1}$.
Suppose first that $p\geq 3$. Then Puig's Conjecture
holds for $\mathcal{G}$, by work of Kessar \cite{Kess3}.
Moreover, $\mathcal{G}$ has the vertex-bounded-defect property
with respect to any $p$-group, by Theorem \ref{thm:vtxbnd}(c).
Hence Feit's Conjecture holds, by Theorem \ref{thm:feitred}.

\medskip
Let now $p=2$, 
and let $D$ be a simple $F\widetilde{\mathfrak{S}}_n$-module.
Since $Z:=\langle z\rangle\leq Z(\widetilde{\mathfrak{S}}_n)$,
in the notation of \ref{noth:schurcov}, is a normal
$2$-subgroup of $\widetilde{\mathfrak{S}}_n$, it acts trivially on $D$. 
Thus there is a simple $F\mathfrak{S}_n$-module $\overline{D}$ such that 
$D=\Inf_Z^{\widetilde{\mathfrak{S}}_n}(\overline{D})$,
where $\Inf$ denotes the inflation from $F\mathfrak{S}_n$-modules
to $F\widetilde{\mathfrak{S}}_n$-modules via the 
normal subgroup $Z\unlhd\widetilde{\mathfrak{S}}_n$.
If $(Q,L)$ is a vertex-source pair of $D$ then $Z\leq Q$ and
$\overline{Q}:=Q/Z$ is a vertex of $\overline{D}$.
Moreover, there is an indecomposable $F\overline{Q}$-module $\overline{L}$ 
such that $L\cong\Inf_Z^Q(\overline{L})$ and such that 
$(\overline{Q},\overline{L})$ is a vertex-source pair of $\overline{D}$,
see \cite[Prop. 2.1]{Ku} and \cite[Prop. 2]{Harr}. Hence we have
$|\mathcal{V}_{\mathcal{G}}(Q)|\leq
 |\mathcal{V}_{\{\mathfrak{S}_n\}}(\overline{Q})|$,
and we are done by (i).

\medskip
(iv) Let $\mathcal{G}=\{\widetilde{\mathfrak{A}}_n\}_{n\geq 1}$. 
Letting again first $p\geq 3$, we may argue as in (ii) to show that
$\mathcal{V}_{\mathcal{G}}(Q)\subseteq
\mathcal{V}_{\{\widetilde{\mathfrak{S}}_n\}}(Q)$,
and we are done using (iii).
Moreover, letting $p=2$, since $Z\leq\widetilde{\mathfrak{A}}_n$,
again using the notation of \ref{noth:schurcov},
we may argue as in (iii) to show that
$|\mathcal{V}_{\mathcal{G}}(Q)|\leq
 |\mathcal{V}_{\{\mathfrak{A}_n\}}(\overline{Q})|$,
and we are done using (ii).

\medskip
(v) Let $\mathcal{G}=\{\mathfrak{B}_n\}_{n\geq 2}$. Then Puig's Conjecture
holds for $\mathcal{G}$, by work of Kessar \cite{Kess2}.
Moreover, $\mathcal{G}$ has the vertex-bounded-defect property
with respect to any $p$-group, by Theorem \ref{thm:vtxbnd}(d) 
and (e). Hence Feit's Conjecture holds, by Theorem \ref{thm:feitred}.

\medskip
(vi) Let $\mathcal{G}=\{\mathfrak{D}_n\}_{n\geq 4}$.
Again suppose first that $p\geq 3$.
Then we may argue as in (ii) to show that
$\mathcal{V}_{\mathcal{G}}(Q)\subseteq\mathcal{V}_{\{\mathfrak{B}_n\}}(Q)$,
and we are done using (v).
Moreover, letting $p=2$, Puig's Conjecture holds for $\mathcal{G}$, 
by work of Kessar \cite{Kess2}, 
and $\mathcal{G}$ has the vertex-bounded-defect property
by Theorem \ref{thm:vtxbnd}(f).
Hence Feit's Conjecture holds, by Theorem \ref{thm:feitred}.
\end{proof}

\begin{rem}\label{rem:dihedral}
\normalfont
We remark that the list of groups in Theorem \ref{thm:summary}
in particular encompasses all infinite series of real
reflection groups, except the groups of type $I_2(m)$,
that is, the dihedral groups $D_{2m}$, where $m\geq 3$.
We give a direct proof that Feit's Conjecture holds for 
$\mathcal{G}=\{D_{2m}\}_{m\geq 3}$ as well;
note that, since $D_{2m}$ is soluble, this also follows 
from much more general work of Puig \cite{Puigpsol,Puig}:

\medskip
Let first $p$ be odd. Then $D_{2m}\cong C_m:C_2$ has a normal
cyclic Sylow $p$-subgroup $C_{m_p}$, where $m_p$ denotes the $p$-part 
of $m$. Hence, by \cite{Erd}, any simple $FD_{2m}$-module
has the normal subgroup $C_{m_p}$ as its vertex, 
and is thus a trivial-source module.
Hence Feit's Conjecture holds for $\mathcal{G}$. 
Note that, by \cite[Thm. 45.12]{Thev}, the source algebras
of the blocks in question are isomorphic to 
$FC_{m_p}$ or $F[C_{m_p}:C_2]\cong FD_{2m_p}$
as interior $C_{m_p}$-algebras,
thus Puig's Conjecture holds for $\mathcal{G}$ as well.

\medskip
Let now $p=2$, and let $D$ be a simple $FD_{2m}$-module.
Then there are two cases:
if $D$ is relatively $C_m$-projective then
$D$ has a normal subgroup $C_{m_2}$ as its vertex, 
and is thus a trivial-source module.
If $D$ is not relatively $C_m$-projective then
its restriction to $C_m\unlhd D_{2m}$ is simple, hence 
one-dimensional, implying again that $D$ is a trivial-source module.
Hence Feit's Conjecture holds for $\mathcal{G}$. 
Note that, since $D_{2m}$ is $2$-nilpotent, by \cite[Prop. 49.13]{Thev}
the blocks in question are nilpotent,
hence, by Puig's Theorem \cite[Thm. 50.6]{Thev},
their source algebras are isomorphic to
$\End_F(iD)\otimes_F FP$, where
$i$ denotes a source idempotent, and
the defect groups in the two cases are
$P=C_{m_2}$ and $P=D_{2m_2}$, respectively;
thus, since $D$ is a trivial-source module, we, moreover, conclude
that $iD$ is the trivial $FP$-module, hence 
the source algebras are isomorphic to $FP$ as interior $P$-algebras,
so that Puig's Conjecture holds for $\mathcal{G}$ as well.
\end{rem}

\medskip
Thus it remains to prove Theorem \ref{thm:vtxbnd}.
To do so, we will often argue along the lines of \cite{DK},
where Theorem \ref{thm:vtxbnd}(a) has already been established.
The key to this line of reasoning is the following:

\begin{rem}\label{rem:knoerr}
\normalfont
By a \textit{Brauer pair} of a group $G$ we understand
a pair $(P,b)$ where $P$ is a $p$-subgroup
of $G$ and $b$ is a block of $F[PC_G(P)]$.
Recall that the \textit{Brauer correspondent} $b^G$, a block of $FG$,
is defined, and if $B=b^G$ then we call $(P,b)$ a \textit{Brauer $B$-pair}. 
Moreover, in the case that $P$ is a defect group of the block $b$ 
we call $(P,b)$ a \textit{self-centralizing} Brauer pair.

\medskip
Let now $\mathcal{G}$ be a set 
of groups, and let $Q$ be
a $p$-group. We say that $\mathcal{G}$ has the 
\textit{strongly-bounded-defect property} with respect to $Q$ if
there is an integer $d_{\mathcal{G}}(Q)$ such that, 
for every group $G$ in $\mathcal{G}$,
the Brauer correspondent $(b_Q)^G$ of any self-centralizing Brauer pair 
$(Q,b_Q)$ of $G$ has defect groups of order at most $d_{\mathcal{G}}(Q)$.

\medskip
By Kn\"orr's Theorem \cite{Kn}, given a block $B$ of $G$,
a self-centralizing Brauer $B$-pair exists, in particular, 
in the case where $Q$ is a vertex of some simple
$FG$-module $M$ belonging to the block $B$.
Hence to prove the vertex-bounded-defect property of $\mathcal{G}$
with respect to a $p$-group $Q$, it suffices to show the 
strongly-bounded-defect property of $\mathcal{G}$ with respect to $Q$, and we
infer $c_{\mathcal{G}}(Q)\leq d_{\mathcal{G}}(Q)$.
We remark that the converse of Kn\"orr's Theorem does not hold, 
see for example Example \ref{expl:cases},
but, to the authors' knowledge, there are no general 
results known towards a characterization of those self-centralizing Brauer 
pairs whose first components actually occur as a vertices of simple modules.

\medskip
Actually, we prove the strongly-bounded-defect property in the Cases 
(a)--(c) of Theorem \ref{thm:vtxbnd},
while for the Cases (d)--(f) we are content with the
weaker vertex-bounded-defect property.
\end{rem}

\medskip
Before we proceed, we give a lemma needed later,
relating Brauer correspondence to covering of blocks.
It should be well known, but we have not been able 
to find a suitable reference.

\begin{lemma}\label{lemma:brauercov}
Let $G$ be a finite group, and let $H\unlhd G$.
Moreover, let $Q\leq H$ be a $p$-subgroup, let $(Q,b)$
be a Brauer pair of $H$, that is, $b$ is a block of $F[QC_H(Q)]$, 
and let $\tilde{b}$ be a block of $F[QC_G(Q)]$ covering $b$. 
Then the Brauer correspondent $\tilde{b}^G$ of $\tilde{b}$ in $G$ 
covers the Brauer correspondent $b^H$ of $b$ in $H$.
\end{lemma}

\begin{proof}
By Passman's Theorem \cite[Thm. 5.5.5]{NT}, we have to show that
$$ \omega_{\tilde{b}^G}(h^{G+})=\omega_{b^H}(h^{G+})
   \quad\text{ for all }h\in H ,$$
where the $\omega$'s are the associated central characters,
$h^G$ is the $G$-conjugacy class of $h\in H$,
and $M^+$ denotes the sum over any subset $M\subseteq G$.
By definition of the Brauer correspondence, and by Passman's Theorem again,
for all $h\in H$, we have
$$ \omega_{b^H}(h^{G+}) 
 \!=\! \omega_b((h^G\cap QC_H(Q))^+) 
 \!=\! \omega_b((h^G\cap QC_G(Q))^+)
 \!=\! \omega_{\tilde{b}}((h^G\cap QC_G(Q))^+) 
 \!=\! \omega_{\tilde{b}^G}(h^{G+}),$$
proving the lemma.
\end{proof}

\section{The Alternating Groups $\mathfrak{A}_n$}\label{sec:alt}

We proceed to prove the bound given in Theorem \ref{thm:vtxbnd}
for the alternating groups. We begin by fixing our notation 
for the Sylow $p$-subgroups of the symmetric and alternating groups, 
respectively; for later use we do this for arbitrary $p$.
Then we focus on the case $p=2$,
collect the necessary facts about the
self-centralizing Brauer pairs of the alternating groups,
and use this to finally prove the desired bound.

\begin{noth}\label{noth:sylowSn}
\normalfont
{\sl Sylow $p$-subgroups.}\, (a) 
We will use the following convention for denoting the 
Sylow $p$-subgroups of $\mathfrak{S}_n$ and $\mathfrak{A}_n$, respectively.
Let $\mathfrak{S}_n$ act on the set $\{1,\ldots,n\}$.
Suppose first that $n=p^m$, for some $m\in\mathbb{N}$.
Moreover, let $C_p:=\langle (1,2,\ldots,p)\rangle$, and set 
$P_1:=1$, $P_p:=C_p$ and 
$P_{p^{i+1}}:=P_{p^i}\wr C_p
=\{(x_1,x_2,\ldots,x_p;\sigma)\mid x_1,\ldots,x_p\in P_{p^i},\, \sigma\in C_p\}$ 
for $i\geq 1$. As usual, for any $i\in\mathbb{N}_0$, we view $P_{p^i}$ 
as a subgroup of $\mathfrak{S}_{p^i}$ in the obvious way.
Then, 
by \cite[4.1.22, 4.1.24]{JK}, $P_{p^m}$ is a Sylow $p$-subgroup of 
$\mathfrak{S}_{p^m}$, and is generated by the following elements, 
where $j=1,\ldots,m$:
\begin{equation}\label{equ:Pgens}
g_j:=\prod_{k=1}^{p^{j-1}}(k,k+p^{j-1},k+2p^{j-1},\ldots, k+(p-1)p^{j-1}).
\end{equation}
For instance, if $p=2$ then $P_8$ is generated by 
$g_1=(1,2),\; g_2=(1,3)(2,4)$, and $g_3=(1,5)(2,6)(3,7)(4,8)$.

\medskip
(b) Next let $n\in\mathbb{N}$ be divisible by $p$, with $p$-adic expansion 
$n=\sum_{j=1}^s \alpha_jp^{i_j}$, for some $s\geq 1$, $i_1>\ldots >i_s\geq 1$,
and $1\leq \alpha_j\leq p-1$ for $j=1,\ldots,s$.
By \cite[4.1.22, 4.1.24]{JK}, $P_n:=\prod_{j=1}^s \prod_{l_j=1}^{\alpha_j}P_{p^{i_j},l_j}$
is then a Sylow $p$-subgroup of $\mathfrak{S}_n$. Here, the direct factor
$P_{p^{i_1},1}$ is acting on $\{1,\ldots,p^{i_1}\}$, $P_{p^{i_1},2}$ is acting
on $\{p^{i_1}+1,\ldots,2p^{i_1}\}$, and so on.
If, finally, $n\geq p+1$ is not divisible by $p$ then
we set $P_n:=P_r$ where $r<n$ is maximal with $p\mid r$.
So, in any case, $P_n$ is a Sylow $p$-subgroup of $\mathfrak{S}_n$.

\medskip
(c) We will examine the case $p=2$ in more detail, as this will be
of particular importance for our subsequent arguments.
As above, suppose that $n$ is even, with $2$-adic expansion
$n=\sum_{j=1}^s2^{i_j}$, for some $s\geq 2$ and 
$i_1>i_2>\ldots >i_s\geq 1$. Letting $n_j:=2^{i_j}$,
we get
$P_n=\prod_{j=1}^s P_{n_j}$,
where 
$P_{n_j}$ is understood to be acting on the set 
$$ \Omega_j:=\left\{(\sum_{l=1}^{j-1}n_l)+1,\ldots,\sum_{l=1}^jn_l\right\} ,$$ 
for $j=1,\ldots,s$. The corresponding generating set for $P_{n_j}$
given by (\ref{equ:Pgens}) will be denoted by 
$\{g_{1,j},\ldots,g_{i_j,j}\}$, for $j=1,\ldots,s$.
So if, for instance, $n=14=8+4+2$ then $P_n=P_{14}$ is generated by 
$g_{1,1}=(1,2)$, $g_{2,1}=(1,3)(2,4)$, $g_{3,1}=(1,5)(2,6)(3,7)(4,8)$, 
$g_{1,2}=(9,10)$, $g_{2,2}=(9,11)(10,12)$, and $g_{1,3}=(13,14)$.

\medskip
(d) We now set $Q_n:=P_n\cap\mathfrak{A}_n$, so that $Q_n$ 
is a Sylow $p$-subgroup of the alternating group $\mathfrak{A}_n$. 
If $p>2$ then clearly $Q_n=P_n$. Thus, suppose again that $p=2$.
If $n=2$ then $Q_n=Q_2=1$.
If $n=2^m$, for some $m\geq 2$, 
then, by (\ref{equ:Pgens}), we obtain the following generators for $Q_n$:
\begin{equation}\label{equ:Qgens}
h_1:=(1,2)(2^{m-1}+1,2^{m-1}+2);\quad \, h_j:=g_j,\text{ for } j=2,\ldots,m.
\end{equation}
For clearly $Q:=\langle h_1,\ldots,h_m\rangle\leq Q_n$, and 
$Q\langle (1,2)\rangle=\langle (1,2)\rangle Q=P_n$. 
Thus $Q=Q_n$.

\medskip
If $n>4$ is even but not a power of $2$ then we again consider 
the $2$-adic expansion $n=\sum_{j=1}^s2^{i_j}$, for some
$s\geq 2$ and some $i_1>\ldots >i_s\geq 1$. Then the 
following elements generate $Q_n$:
\begin{equation}\label{equ:Qgens2}
h_{1,j}:=g_{1,s}g_{1,j},\text{ for } j=1,\ldots,s-1;\quad
h_{k,j}:=g_{k,j}, \text{ for } j=1,\ldots,s \text{ and } k=2,\ldots,i_j.
\end{equation}
Namely, these elements generate a subgroup $Q$ of $Q_n$ such that 
$Q\langle g_{1,s}\rangle=\langle g_{1,s}\rangle Q=P_n$.
For instance, 
$Q_{14}=P_{14}\cap\mathfrak{A}_{14}
=(P_8\times P_4\times P_2)\cap\mathfrak{A}_{14}$ is 
generated by the elements 
$h_{1,1}=(1,2)(13,14)$, $h_{1,2}=(9,10)(13,14)$,
$h_{2,1}=(1,3)(2,4)$, $h_{3,1}=(1,5)(2,6)(3,7)(4,8)$,
and $h_{2,2}=(9,11)(10,12)$.
\end{noth}

\medskip 
For the remainder of this section, let $p=2$.

\begin{prop}\label{prop:centre}
Let $n=\sum_{j=1}^s2^{i_j}\geq 2$ be the $2$-adic expansion of $n$,
where $s\in\mathbb{N}$ and $i_1>\ldots >i_s\geq 1$, and let again 
$n_j:=2^{i_j}$ for $j=1,\ldots,s$.

\quad {\rm (a)}\, If $n\equiv 0\pmod{4}$ then 
$$C_{\mathfrak{S}_n}(Q_n)=C_{\mathfrak{A}_n}(Q_n)=Z(Q_n)=\begin{cases}
Q_4,& \text{if } n=4\\
Z(P_n)=Z(P_{n_1})\times\cdots\times Z(P_{n_s}),&\text{if } n>4.
\end{cases}$$

\quad {\rm (b)}\, If $n\equiv 2\pmod{4}$ then $i_s=1$, and
$$C_{\mathfrak{S}_n}(Q_n)=Z(P_n)=Z(P_{n_1})\times\cdots\times Z(P_{n_s})
=Z(Q_n)\times P_2=C_{\mathfrak{A}_n}(Q_n)\times P_2.$$
\end{prop}

\begin{proof}
We may assume that $n\geq 4$. Then $\Omega_1,\ldots,\Omega_s$ 
are the orbits of $P_n$ on $\{1,\ldots,n\}$,
as well as the orbits of $Q_n$ on $\{1,\ldots,n\}$, 
where $\Omega_j$ is as above in \ref{noth:sylowSn}(c). 
Since $|\Omega_1|>\ldots >|\Omega_s|$, the
$Q_n$-sets $\Omega_1,\ldots,\Omega_s$ are pairwise
non-isomorphic. For $j=1,\ldots,s$, let $\omega_j\in\Omega_j$,
and set $R_j:=\Stab_{Q_n}(\omega_j)$. Then $\Omega_j$ is as
$Q_n$-set isomorphic to $Q_n/R_j$, and we have the following 
group isomorphism, see \cite[La. 4.3]{DK}: 
$$\prod_{j=1}^sN_{Q_n}(R_j)/R_j\longrightarrow C_{\mathfrak{S}_n}(Q_n).$$
In particular, $C_{\mathfrak{S}_n}(Q_n)$ is a $2$-group and, hence, so is
$Q_n C_{\mathfrak{S}_n}(Q_n)$.
Thus there is some $g\in\mathfrak{S}_n$ such that 
${}^g(Q_n C_{\mathfrak{S}_n}(Q_n))\leq P_n$.
In particular, we have ${}^g Q_n\leq P_n\cap\mathfrak{A}_n=Q_n$,
that is, $g\in N_{\mathfrak{S}_n}(Q_n)$. Hence 
we have $g\in N_{\mathfrak{S}_n}(C_{\mathfrak{S}_n}(Q_n))$ as well,
implying $C_{\mathfrak{S}_n}(Q_n)\leq P_n$, and thus
$C_{\mathfrak{S}_n}(Q_n)=C_{P_n}(Q_n)$.
So it suffices to show that
$$C_{P_n}(Q_n)=\begin{cases}
Q_4,&\text{if } n=4\\
Z(P_n),&\text{if } n\neq 4,
\end{cases}$$
since then we also get 
$$C_{\mathfrak{A}_n}(Q_n)=Z(Q_n)=\begin{cases}
Q_4,&\text{ if } n=4\\
Z(P_n),&\text{ if } 4<n\equiv 0\pmod{4}\\
Z(P_{n_1})\times\cdots\times Z(P_{n_{s-1}}),&\text{ if } n\equiv 2\pmod{4}.
\end{cases}$$ 

\medskip
The statement for $n=4$ is clear. Next suppose that $n=2^m$, 
for some $m\geq 3$. We argue with induction on $m$, and
show that $Z(P_n)=C_{P_n}(Q_n)$. 
For $m=3$ this is immediately checked to be true, 
so that we may now suppose that $m>3$.
We consider $P_n$ again as the wreath product 
$P_{2^{m-1}}\wr C_2
=\{(x_1,x_2;\sigma)\mid x_1,\,x_2\in P_{2^{m-1}},\, \sigma\in C_2\}$.
Let $x:=(x_1,x_2;\sigma)\in C_{P_n}(Q_n)$, so that, 
for each $y:=(y_1,y_2;\pi)\in Q_n$, we have
\begin{equation}\label{equ:cent}
(x_1y_{\sigma(1)},x_2y_{\sigma(2)};\sigma\pi)
=(x_1,x_2;\sigma)(y_1,y_2;\pi)
=(y_1,y_2;\pi)(x_1,x_2;\sigma)
=(y_1x_{\pi(1)},y_2x_{\pi(2)};\pi\sigma).
\end{equation}
Setting $y_1:=y_2:=1$ and $\pi:=(1,2)$, Equation (\ref{equ:cent}) 
yields $x_1=x_2$. Next we set $\pi:=1$, $y_1:=1$,
and $1\neq y_2\in Q_{2^{m-1}}$. Then (\ref{equ:cent}) this time 
implies $x_1y_{\sigma(1)}=x_1$ and $x_1y_{\sigma(2)}=y_2x_1$.
Therefore $\sigma=1$ and $x_1\in C_{P_{2^{m-1}}}(Q_{2^{m-1}})$. 
Thus, by induction, 
$x_1\in Z(P_{2^{m-1}})$. Consequently, $x=(x_1,x_1;1)\in Z(P_n)$,
and we have $C_{P_n}(Q_n)=Z(P_n)\leq Q_n$. 

\medskip
Now let $n>4$ with $s\geq 2$. We show that also in this case 
$C_{P_n}(Q_n)=Z(P_n)=Z(P_{n_1})\times\cdots\times Z(P_{n_s}).$
For this, let $x\in C_{P_n}(Q_n)$, and write $x=x_1\cdots x_s$ 
for appropriate $x_j\in P_{n_j}$ and $j=1,\ldots,s$. Since
$Q_{n_1}\times\cdots\times Q_{n_j}\leq Q_n$, we deduce 
that $x_j\in C_{P_{n_j}}(Q_{n_j})$, for $j=1,\ldots,s$. 
Hence, by what we have just proved above, $x_j\in Z(P_{n_j})$ if $i_j>2$. 
Moreover, $x_j\in Z(Q_4)=Q_4$ if $i_j=2$, and clearly 
$x_j\in Z(P_2)=P_2$ if $i_j=1$. 
Suppose that there is some $j\in\{1,\ldots,s\}$ with $i_j=2$. 
Then $j\in\{s-1,s\}$. We need to show that $x_j\in Z(P_4)$.
Assume that this is not the case. In the notation of 
\ref{noth:sylowSn}, we may then suppose that $x_j=h_{i_j,j}=g_{2,j}$.
But this leads to the contradiction 
$x(g_{1,j}g_{1,s})x^{-1}=x_jg_{1,j}x_j\cdot g_{1,s}\neq g_{1,j}g_{1,s}$
if $j=s-1$, and to the contradiction 
$x(g_{1,1}g_{1,j})x^{-1}=g_{1,1}\cdot x_jg_{1,j}x_j\neq g_{1,1}g_{1,j}$ 
if $j=s$. Thus also $x_j\in Z(P_4)$, and we have shown that
$Z(P_n)\leq C_{P_n}(Q_n)\leq Z(P_{n_1})\times\cdots\times 
Z(P_{n_s})=Z(P_n)$.
\end{proof}

\begin{noth}\label{noth:fix}
\normalfont
{\sl The $2$-Blocks of $\mathfrak{A}_n$.}\, 
(a) Recall from \cite[6.1.21]{JK}
that each block $B$ of $F\mathfrak{S}_n$ can be
labelled combinatorially by some integer $w\geq 0$ and a $2$-regular
partition $\kappa$ of $n-2w$. We call $w$ the \textit{$2$-weight} of $B$, and
$\kappa$ the \textit{$2$-core} of $B$.
Moreover, by \cite[Thm. 6.2.39]{JK}, the defect groups of $B$ are in 
$\mathfrak{S}_n$ conjugate 
to $P_{2w}\leq\mathfrak{S}_{2w}\leq\mathfrak{S}_n$.

\medskip
(b) 
The following relationships 
between blocks of $F\mathfrak{S}_n$ and 
$F\mathfrak{A}_n$ are well known; see for instance \cite{O1}:
for any partition $\lambda$ of $n$, we denote its conjugate 
partition by $\lambda'$. That is, the Young diagram $[\lambda']$ of
$\lambda'$ is obtained by transposing the Young diagram $[\lambda]$.

\medskip
Suppose now that $B$ is a block of $F\mathfrak{S}_n$ of weight $w$ and 
with $2$-core $\kappa$. Denote the corresponding block idempotent 
of $F\mathfrak{S}_n$ by $e_B$. The $2$-core
$\kappa$ is a triangular partition, so that
$\kappa=\kappa'$. If $w\geq 1$ then
$e_B$ is a block idempotent of $F\mathfrak{A}_n$, and if $w=0$ then 
$e_B=e_{B'}+ e_{B''}$ for $\mathfrak{S}_n$-conjugate blocks $B'\neq B''$ 
of $F\mathfrak{A}_n$ of defect 0; note that for $w=1$ this also
yields a (single) block of $F\mathfrak{A}_n$ of defect 0.
Hence, the \textit{weight} of any block of $F\mathfrak{A}_n$ 
is understood to be the weight $w$ of the covering block of $F\mathfrak{S}_n$,
and the associated defect groups are in $\mathfrak{A}_n$ conjugate 
to $Q_{2w}\leq\mathfrak{A}_{2w}\leq\mathfrak{A}_n$.

\medskip
(c)
Let $B$ be a block of $F\mathfrak{A}_n$
of weight $w\geq 0$. Let further $(Q,b_Q)$ be a
self-centralizing Brauer $B$-pair. 
Then, by \cite[Thm. 5.5.21]{NT}, there
is a defect group $P$ of $B$ such that
$Z(P)\leq C_P(Q)\leq Q\leq P$.
Replacing $(Q,b_Q)$ by a suitable $\mathfrak{A}_n$-conjugate,
we may assume that $P=Q_{2w}$. 
If $w$ is even then, by Proposition \ref{prop:centre},
$Z(Q_{2w})$ has no fixed points on
$\{1,\ldots,2w\}$, hence, in this case, 
$Q$ acts fixed point freely on precisely $2w$ points;
note that this also holds for $w=0$.
If 
$w$ is odd then, by Proposition \ref{prop:centre}
again, $Z(Q_{2w})$ has exactly the two fixed points 
$\{2w-1,2w\}$ on $\{1,\ldots,2w\}$,
hence, in this case $Q$ acts fixed point freely precisely 
on either $\{1,\ldots,2w-2\}$ or $\{1,\ldots,2w\}$, that is,
on $2x$ points, where $w-1\leq x\leq w$;
note that for $w=1$ we have $x=0$. 

\medskip
The following example shows that even if we restrict ourselves
to Brauer pairs arising from vertices of simple modules
we have to deal with both cases $w-1\leq x\leq w$:
\end{noth}

\begin{expl}\label{expl:fix}
\normalfont
Suppose that $n=2m$, for some odd integer $m\geq 3$, 
and consider the simple $F\mathfrak{S}_n$-module $D^{(m+1,m-1)}$
labelled by the partition $(m+1,m-1)$ of $n$. 
This is the {\it basic spin} $F\mathfrak{S}_n$-module, belonging
to the principal block of $F\mathfrak{S}_n$, which has
weight $m$. Since $n\equiv 2\pmod{4}$, the restriction
$\res_{\mathfrak{A}_n}^{\mathfrak{S}_n}(D^{(m+1,m-1)})=:E^{(m+1,m-1)}$ 
is simple, by \cite{Ben}, and thus belongs to the principal block $B_0$ of 
$F\mathfrak{A}_n$.
By \cite[Thm. 7.2]{DKspin}, $E^{(m+1,m-1)}$ has common vertices with the
basic spin $F\mathfrak{S}_{n-1}$-module $D^{(m,m-1)}$. Therefore, the vertices 
of $E^{(m+1,m-1)}$ are conjugate to subgroups of $Q_{n-2}$ 
and have, in particular, fixed points on $\{1,\ldots,n\}$, 
while $Q_n$ acts of course fixed point freely. 
This shows that there is indeed a self-centralizing Brauer
$B_0$-pair $(Q,b_Q)$ of $\mathfrak{A}_n$,
where $Q$ arises as a vertex of a simple $F\mathfrak{A}_n$-module and
such that $Q$ has strictly more fixed points on $\{1,\ldots,n\}$ than the 
associated defect group $Q_n$ of its Brauer correspondent 
$b_Q^{\mathfrak{A}_n}=B_0$.
\end{expl}

\medskip
The next theorem is motivated by the results of \cite{DK},
where the self-centralizing Brauer pairs of the symmetric groups
are examined, for which, using the above notation, we 
necessarily have $x=w$. We pursue the analogy to the case
of the symmetric groups as far as possible, the
treatment being reminiscent of the exposition in \cite[Sect. 1]{O}. 

\begin{thm}\label{thm:self}
Let $(Q,b_Q)$ be a self-centralizing Brauer pair of $\mathfrak{A}_n$. 

\quad {\rm (a)}\,
Let $\Omega\subseteq\{1,\ldots,n\}$ be such that
$Q$ acts fixed point freely on $\Omega$ and fixes
$\{1,\ldots,n\}\smallsetminus\Omega$ pointwise. 
Then we have $C_{\mathfrak{A}(\Omega)}(Q)=Z(Q)$.

\quad {\rm (b)}\,
Let $P$ be a defect group of the Brauer correspondent 
$B:=b_Q^{\mathfrak{A}_n}$ of $b_Q$ in $\mathfrak{A}_n$
such that $C_P(Q)\leq Q\leq P$,
and let $\hat{\Omega}\subseteq\{1,\ldots,n\}$ be such that
$P$ acts fixed point freely on $\hat{\Omega}$ and fixes
$\{1,\ldots,n\}\smallsetminus\hat{\Omega}$ pointwise.
Then we even have $C_{\mathfrak{A}(\hat{\Omega})}(Q)=Z(Q)$.
\end{thm}

\begin{proof}
Since $(Q,b_Q)$ is self-centralizing, 
the block $b_Q$ of $F[QC_{\mathfrak{A}_n}(Q)]$ has defect group $Q$.
Let $w\geq 0$ be the weight of $B=b_Q^{\mathfrak{A}_n}$.
By the observations made in \ref{noth:fix}, 
we have $2w-2\leq 2x=|\Omega|\leq 2w=|\hat{\Omega}|$, 
and we may suppose that $\Omega=\{1,\ldots,2x\}$ and 
$\hat{\Omega}=\{1,\ldots,2w\}$, that is,
$Q\leq\mathfrak{A}_{2x}\leq\mathfrak{A}_{2w}$ and
$P=Q_{2w}\leq\mathfrak{A}_{2w}$.
We have 
$C_{\mathfrak{S}_n}(Q)=C_{\mathfrak{S}_{2x}}(Q)\times\mathfrak{S}_{n-2x}$, 
and thus also 
$QC_{\mathfrak{S}_n}(Q)=QC_{\mathfrak{S}_{2x}}(Q)\times\mathfrak{S}_{n-2x}$. 
Consider the following chain of normal subgroups
\begin{equation}\label{equ:chn}
QC_{\mathfrak{A}_{2x}}(Q)\times\mathfrak{A}_{n-2x}
  \unlhd QC_{\mathfrak{A}_n}(Q)\unlhd QC_{\mathfrak{S}_n}(Q)
  =QC_{\mathfrak{S}_{2x}}(Q)\times\mathfrak{S}_{n-2x}.
\end{equation}

\medskip
Since $|QC_{\mathfrak{S}_n}(Q):QC_{\mathfrak{A}_n}(Q)|\leq 2$, 
by \cite[Cor. 5.5.6]{NT} there is a unique block $\tilde{b}_Q$ of
$F[QC_{\mathfrak{S}_n}(Q)]$ covering $b_Q$. 
In particular, $(Q,\tilde{b}_Q)$ is a (not necessarily self-centralizing)
Brauer pair of $\mathfrak{S}_n$. 
We may write
$\tilde{b}_Q=\tilde{b}_0\otimes\tilde{b}_1$, for some block $\tilde{b}_0$
of $F[QC_{\mathfrak{S}_{2x}}(Q)]$ and some block $\tilde{b}_1$
of $F\mathfrak{S}_{n-2x}$. Since $Q$ has no fixed points on $\{1,\ldots,2x\}$,
by \cite[Prop. 1.2, Prop. 1.3]{O} we conclude that
$F[QC_{\mathfrak{S}_{2x}}(Q)]$ has only one block,
that is, the principal one.
Therefore, each block of $F[QC_{\mathfrak{A}_{2x}}(Q)]$ is
covered by the principal block
$\tilde{b}_0$ of $F[QC_{\mathfrak{S}_{2x}}(Q)]$. Hence all blocks of
$F[QC_{\mathfrak{A}_{2x}}(Q)]$ are conjugate in $QC_{\mathfrak{S}_{2x}}(Q)$.
But then also $F[QC_{\mathfrak{A}_{2x}}(Q)]$ has only one block,
that is, the principal block $b_0$.

\medskip
Moreover, $b_Q$ covers some block $b_0\otimes b_1$ of 
$F[QC_{\mathfrak{A}_{2x}}(Q)\times\mathfrak{A}_{n-2x}]$,
where $b_1$ is a block of $F\mathfrak{A}_{n-2x}$,
and, by \cite[Cor. 5.5.6]{NT} again, $b_Q$ is the
unique block of $F[QC_{\mathfrak{A}_n}(Q)]$ covering $b_0\otimes b_1$. 
The defect groups of $b_0\otimes b_1$ are in $QC_{\mathfrak{A}_n}(Q)$ 
conjugate to subgroups of $Q$, by Fong's Theorem \cite[Cor. 5.5.16]{NT}.
Hence $b_1$ has to be a block of defect 0. Thus, since $\tilde{b}_1$ 
covers $b_1$ we infer from \ref{noth:fix} that $\tilde{b}_1$
is a block of weight $\tilde{w}=0$ or $\tilde{w}=1$, that is, of defect
$0$ or $1$, respectively.
Moreover, since $b_0$ is the principal block of 
$F[QC_{\mathfrak{A}_{2x}}(Q)]$,
we infer that $Q\in\syl_2(QC_{\mathfrak{A}_{2x}}(Q))$.

\medskip
Thus we may summarize the properties of the relevant blocks 
of the subgroups in (\ref{equ:chn}) as follows:
$$\begin{array}{|rcl|c|rcrcl|}
\hline
b_0 & \otimes & b_1 &
b_Q &
\tilde{b}_Q & = & \tilde{b}_0 & \otimes & \tilde{b}_1 \\
\hline
\text{defect }Q &&\text{defect }0& \text{defect } Q &&&&& 
\text{weight }\tilde{w}\in\{0,1\} \\
\text{principal} &&&&&& \text{principal} && \\
\hline
\end{array}$$

\medskip
To show (a), assume that 
$QC_{\mathfrak{A}_{2x}}(Q)$ is not a $2$-group, so that there is
some $1\neq g\in QC_{\mathfrak{A}_{2x}}(Q)$ of odd order. Thus 
$g\in C_{\mathfrak{A}_{2x}}(Q)$, and we denote the 
conjugacy class of $g$ in  $QC_{\mathfrak{A}_{2x}}(Q)$ by $C$. 
Since $Q\in\syl_2(QC_{\mathfrak{A}_{2x}}(Q))$, we also have
$Q\in \syl_2(C_{QC_{\mathfrak{A}_{2x}}(Q)}(g))$. In particular, 
$Q$ is a defect group of the conjugacy class $C$ and of the 
conjugacy class $\{1\}\neq C$. Hence, from \cite[Cor. IV.4.17]{Feit}
we infer that $F[QC_{\mathfrak{A}_{2x}}(Q)]$ has two blocks of maximal defect, 
contradicting the fact that the principal block is the only 
block of $F[QC_{\mathfrak{A}_{2x}}(Q)]$. 
Hence $QC_{\mathfrak{A}_{2x}}(Q)$ is a $2$-group, and
from this we finally deduce that $QC_{\mathfrak{A}_{2x}}(Q)=Q$,
that is, $C_{\mathfrak{A}_{2x}}(Q)=Z(Q)$, proving (a).

\medskip
To show (b) we may suppose that $x=w-1$, hence we have $n\geq 2x+2$. 
We show that in this case we have
$$ QC_{\mathfrak{A}_{n}}(Q)
  =QC_{\mathfrak{A}_{2x}}(Q)\times\mathfrak{A}_{n-2x},$$
from which we get $C_{\mathfrak{A}_{2w}}(Q)=QC_{\mathfrak{A}_{2x}}(Q)=Z(Q)$.
So assume, for a contradiction, that 
$QC_{\mathfrak{A}_{2x}}(Q)\times\mathfrak{A}_{n-2x}<QC_{\mathfrak{A}_{n}}(Q)$.
We consider again the chain of subgroups (\ref{equ:chn}), where 
$[QC_{\mathfrak{S}_{n}}(Q):
 (QC_{\mathfrak{A}_{2x}}(Q)\times\mathfrak{A}_{n-2x})]\leq 4$.
Since $QC_{\mathfrak{A}_{n}}(Q)\leq\mathfrak{A}_{n}$ but
$\mathfrak{S}_{n-2x}\not\leq\mathfrak{A}_{n}$, we get
$[QC_{\mathfrak{S}_{n}}(Q):QC_{\mathfrak{A}_{n}}(Q)]=2$, and hence 
$[QC_{\mathfrak{A}_{n}}(Q):
  (QC_{\mathfrak{A}_{2x}}(Q)\times\mathfrak{A}_{n-2x})]=2$.

\medskip
Since $b_Q$ has defect group $Q$, from
Fong's Theorem \cite[Thm. 5.5.16]{NT}
we conclude that the inertial group of $b_0\otimes b_1$ is given as
$T_{QC_{\mathfrak{A}_{n}}(Q)}(b_0\otimes b_1)
 =QC_{\mathfrak{A}_{2x}}(Q)\times\mathfrak{A}_{n-2x}$.
Thus $b_0\otimes b_1$ is not $QC_{\mathfrak{A}_{n}}(Q)$-invariant,
hence is not $QC_{\mathfrak{S}_{n}}(Q)$-invariant either.
Since $b_0$ is the principal block of $F[QC_{\mathfrak{A}_{2x}}(Q)]$,
this implies that $b_1$ is not $\mathfrak{S}_{n-2x}$-invariant,
from which, by \ref{noth:fix}, we infer that $\tilde{b}_1$ 
has weight $\tilde{w}=0$, that is, $\tilde{b}_1$ has defect $0$.

\medskip
Let $\tilde{B}:=\tilde{b}_Q^{\mathfrak{S}_{n}}$ be the 
Brauer correspondent of 
$\tilde{b}_Q=\tilde{b}_0\otimes\tilde{b}_1$ in $\mathfrak{S}_{n}$.
Since $\tilde{b}_1$ has weight $\tilde{w}=0$
and $Q$ acts fixed point freely on the set $\Omega$ of cardinality $2x$, 
we conclude from \cite[Thm. 1.7]{O} that $\tilde{B}$ is a block 
of weight $x$. But, by \ref{noth:fix}, the block of $F\mathfrak{S}_{n}$ covering $B$, 
has weight $w$, contradicting Lemma \ref{lemma:brauercov}.
Thus we have 
$QC_{\mathfrak{A}_{2x}}(Q)\times\mathfrak{A}_{n-2x}=QC_{\mathfrak{A}_{n}}(Q)$,
proving (b).
\end{proof}

\begin{rem}\label{rem:self}
\normalfont
A closer analysis of the arguments in the proof of 
Theorem \ref{thm:self} yields a somewhat more precise
description of the self-centralizing Brauer pairs of $\mathfrak{A}_n$,
which at least helps to exclude certain subgroups from being
the first component of a self-centralizing Brauer pair.

\medskip
It turns out that there are only the cases listed below, 
all of which, by Example \ref{expl:cases}, actually occur.
We give the results without proofs, since we do not use these facts later on.
Note that, by Brauer's Third Main Theorem \cite[Thm. 5.6.1]{NT},
$b_Q$ is the principal block if and only if its
Brauer correspondent $b_Q^{\mathfrak{A}_n}$ is, 
which in turn is equivalent to $n\in\{2w,2w+1\}$.

\medskip
(a) Let first $n\in\{2x,2x+1\}$, thus we have $x=w$, and
hence $QC_{\mathfrak{A}_{n}}(Q)=QC_{\mathfrak{A}_{2x}}(Q)$ anyway.
Moreover, $b_Q$ is principal, hence, in particular,
is $QC_{\mathfrak{S}_{n}}(Q)$-invariant. 
It turns out that always $\tilde{w}=0$, and that both cases
$d:=[C_{\mathfrak{S}_{2x}}(Q):C_{\mathfrak{A}_{2x}}(Q)]\in\{1,2\}$ occur.

\medskip
(b) Now let $n\geq 2x+2$. Then it turns out that $w=x+\tilde{w}$ and that
$$ d:=[C_{\mathfrak{S}_{2x}}(Q):C_{\mathfrak{A}_{2x}}(Q)]
     =[QC_{\mathfrak{A}_{n}}(Q):
      (QC_{\mathfrak{A}_{2x}}(Q)\times\mathfrak{A}_{n-2x})]\in\{1,2\} .$$
Moreover, only the following cases occur:
\begin{center}
\begin{tabular}{|l|l|l|l|l|}
\hline
$w$ odd  & $d=1$ & $\tilde{w}=1$ 
         & $b_Q$ is $QC_{\mathfrak{S}_{n}}(Q)$-invariant \\
         & $d=2$ & $\tilde{w}=0$ 
         & $b_Q$ is $QC_{\mathfrak{S}_{n}}(Q)$-invariant \\
\hline
$w$ even & $d=1$ & $\tilde{w}=0$
         & $b_Q$ is not $QC_{\mathfrak{S}_{n}}(Q)$-invariant \\
         & $d=2$ & $\tilde{w}=0$
         & $b_Q$ is $QC_{\mathfrak{S}_{n}}(Q)$-invariant \\
\hline
\end{tabular}
\end{center}
Note that $b_Q$ is principal if and only if 
$n\in\{2x+2,2x+3\}$ where $x$ is even and $d=1$.
\end{rem}

\medskip
We are now prepared to prove the bound given in Theorem \ref{thm:vtxbnd}.

\begin{prop}\label{prop:boundAn}
Let $B$ be a block of $F\mathfrak{A}_n$
with defect group $P$. Let further 
$(Q,b_Q)$ be a self-centralizing Brauer $B$-pair. Then we have
$$|P|\leq \frac{(|Q|+2)!}{2}.$$
\end{prop}

\begin{proof}
Let $w$ be the weight of $B$.
By Theorem \ref{thm:self}, there is a subset $\Omega$ of $\{1,\ldots,n\}$ 
with $2w-2\leq 2x=|\Omega|\leq 2w$,
where $x$ is as in \ref{noth:fix},
such that $Q$ acts fixed point freely on
$\Omega$ and fixes $\{1,\ldots,n\}\smallsetminus\Omega$ pointwise. 
Moreover, $C_{\mathfrak{A}(\Omega)}(Q)=Z(Q)$. 

\medskip
Assume there are $Q$-orbits $\Omega'=\{\omega'_1,\ldots,\omega'_m\}$ 
and $\Omega''=\{\omega''_1,\ldots,\omega''_m\}$ on $\Omega$
that are isomorphic as $Q$-sets. Then we may suppose that there is an 
isomorphism of $Q$-sets mapping $\omega'_i$ to $\omega''_i$, 
for $i=1,\ldots,m$. Since, by our assumption, $Q$ acts fixed point 
freely on $\Omega$, we deduce that $m\geq 2$ is even, and therefore
the permutation $(\omega'_1,\omega''_1)\cdots (\omega'_m,\omega''_m)$
is contained in $C_{\mathfrak{A}(\Omega)}(Q)$. 
But, on the other hand, the elements in $Z(Q)$ have to fix every 
$Q$-orbit, so that 
$(\omega_1,\omega_1')\cdots (\omega_m,\omega_m')\notin Z(Q)$, 
a contradiction.

\medskip
Hence we deduce that $\Omega=\biguplus_{i=1}^k\Omega_i$
consists of pairwise non-isomorphic $Q$-orbits.
For $j=1,\ldots,k$, let $\omega_j\in\Omega_j$,
and set $R_j:=\Stab_{Q}(\omega_j)$, where 
we may choose notation such that $|R_1|\leq \ldots\leq |R_k|$. 
Then, by \cite[La. 4.3]{DK},
we have the group isomorphism
$$\varphi:\prod_{j=1}^kN_Q(R_j)/R_j\longrightarrow 
  C_{\mathfrak{S}(\Omega)}(Q) ,$$
which can be described as follows:
for $i=1,\ldots,k$, let $\pi_i:Q\longrightarrow\mathfrak{S}(\Omega_i)\cap Q$
be the canonical projection, and let 
$Z_i:=C_{\mathfrak{S}(\Omega_i)}(\pi_i(Q))$. Then
$\varphi(N_Q(R_i)/R_i)=Z_i$, for $i=1,\ldots,k$. Hence
$C_{\mathfrak{S}(\Omega)}(Q)=Z_1\times\cdots\times Z_k$, and
$Z(Q)=C_{\mathfrak{A}(\Omega)}(Q)
 =(Z_1\times\cdots\times Z_k)\cap\mathfrak{A}(\Omega)$.
Since $Q$ acts fixed point freely on $\Omega$, 
we have $R_i<Q$, thus $N_Q(R_i)>R_i$, in particular 
$|Z_i|\geq 2$. 
Since $(Z_2\times\cdots\times Z_k)\cap\mathfrak{A}(\Omega)\leq Z(Q)\leq Q$ 
acts trivially on $\Omega_1$, we have 
$(Z_2\times\cdots\times Z_k)\cap\mathfrak{A}(\Omega)\leq R_1$.
Moreover, as we have just mentioned, 
$|Z_2\times\cdots\times Z_k|\geq 2^{k-1}$, so that
for $k\geq 2$ we have
$|R_1|\geq|(Z_2\times\cdots\times Z_k)\cap\mathfrak{A}(\Omega)|\geq 2^{k-2}$. 
Therefore, for $k\geq 2$ we get 
$$|\Omega|=\sum_{i=1}^k|Q:R_i|
\leq\frac{k|Q|}{|R_1|}\leq \frac{k|Q|}{2^{k-2}} .$$ 

\medskip
Thus, for $k\geq 4$ we infer $|\Omega|\leq |Q|$. It remains 
to consider the cases $k\leq 3$: if $k=1$ then we have $|\Omega|\leq |Q|$
anyway. If $k=2$ then we have $|\Omega|\leq |Q|$, except if 
$\Omega_1$ is the regular $Q$-orbit, that is, $R_1=\{1\}$, and hence 
$Z_1\cong Q$. Now $Z_1\cap\mathfrak{A}(\Omega)\leq R_2$ entails
$$ \frac{|Q|}{2}\leq |Z_1\cap\mathfrak{A}(\Omega)|
   \leq |R_2|\leq\frac{|Q|}{2} ,$$
thus $|R_2|=|Q|/2$, hence $|Z_2|=2$, implying $|\Omega|=|Q|+2$.
Moreover, we have $|Z_1\cap\mathfrak{A}(\Omega)|=|Q|/2$, hence 
$Z_1\not\leq\mathfrak{A}(\Omega)$. Since $Z_1$ acts regularly
on $\Omega_1$ and fixes $\Omega_2$ pointwise, we infer that
$Z_1$ contains an $|\Omega_1|$-cycle, that is, $Z_1\cong Q$ is cyclic.
Note that we have $|Q|\geq 4$, implying that
$2x=|\Omega|=|Q|+2\equiv 2\pmod{4}$, that is, $x$ is odd, 
hence, by \ref{noth:fix}, we infer that $w=x$.

\medskip
Next we observe that if $R_1=\{1\}$ then $k\geq 2$ forces $k=2$, since
$1=|R_1|\geq |(Z_2\times\cdots\times Z_k)\cap\mathfrak{A}(\Omega)|
 \geq 2^{k-2}$. 
Hence, if $k=3$ then we have $|\Omega|\leq |Q|$, except if 
$|R_1|=|R_2|=2$. Thus from 
$|(Z_2\times Z_3)\cap\mathfrak{A}(\Omega)|\geq 2$ and
$|(Z_1\times Z_3)\cap\mathfrak{A}(\Omega)|\geq 2$ we deduce
$R_1=(Z_2\times Z_3)\cap\mathfrak{A}(\Omega)\leq Z(Q)$ and
$R_2=(Z_1\times Z_3)\cap\mathfrak{A}(\Omega)\leq Z(Q)$,
showing $|Z_1|=|Q|/2=|Z_2|$. This yields
$$ |Q|\geq |(Z_1\times Z_2\times Z_3)\cap\mathfrak{A}(\Omega)|\geq 
   \frac{|Q|}{2}\cdot\frac{|Q|}{2}\cdot 2\cdot\frac{1}{2}=\frac{|Q|^2}{4} ,$$ 
hence $|Q|\leq 4$. Therefore we have $Q\cong V_4$, and 
$|\Omega_1|=|\Omega_2|=|\Omega_3|=2$, thus $2x=|\Omega|=6=|Q|+2$.
Note that $x=3$, by \ref{noth:fix}, implies that $w=x=3$ as well.

\medskip
Consequently, in any case we get $|\Omega|\leq |Q|$
or, in the two exceptions, $|\Omega|=|Q|+2$ and $x=w$.
This implies $2w\leq |Q|+2$ and, since $P\leq\mathfrak{A}_{2w}$,
we get $|P|\leq\frac{(|Q|+2)!}{2}$.
\end{proof}

\begin{expl}\label{expl:cases}
\normalfont
We give a few examples, found with the help of the 
computer algebra system {\sf GAP} \cite{GAP}, 
showing that all the cases listed in Remark \ref{rem:self} actually occur. 
In particular, the exceptional cases detected in the proof of Proposition 
\ref{prop:boundAn}, namely $Q$ cyclic with two orbits of lengths $|Q|$ 
and $2$, as well as $Q\cong V_4$ with three orbits of length $2$ each,
occur for the principal block of $\mathfrak{A}_6$.

\medskip
(a)
The principal blocks of $\mathfrak{A}_4$ and of $\mathfrak{A}_5$
both have weight $w=2$, their defect groups are abelian and
conjugate to $Q_4\leq\mathfrak{A}_4$ and, in each case, up to conjugacy, 
there is a unique self-centralizing Brauer pair:
$$\begin{array}{|ll|r|r|r|r|}
\hline
Q \quad (n\in\{4,5\},\,\,\, w=2) & & |Q| & |Z(Q)| & x & d \\
\hline
\hline
\langle(1,2)(3,4),(1,3)(2,4)\rangle & \cong V_4 & 4 & 4 & 2 & 1 \\
\hline
\end{array}$$

\medskip
The non-principal blocks of $\mathfrak{A}_7$ and of $\mathfrak{A}_{10}$ 
both have weight $w=2$, their defect groups are abelian and conjugate 
to $Q_4\leq\mathfrak{A}_4$ and, in each case, up to conjugacy, 
there are two self-centralizing Brauer pairs:
$$\begin{array}{|ll|r|r|r|r|}
\hline
Q \quad (n\in\{7,10\},\,\,\, w=2) & 
  & |Q| & |Z(Q)| & x & d \\
\hline
\hline
\langle(1,2)(3,4),(1,3)(2,4)\rangle & \cong V_4 & 4 & 4 & 2 & 1 \\
\langle(1,2)(3,4),(1,3)(2,4)\rangle & \cong V_4 & 4 & 4 & 2 & 1 \\
\hline
\end{array}$$

\medskip
(b)
The principal blocks of $\mathfrak{A}_6$ and of $\mathfrak{A}_7$,
and the non-principal block of $\mathfrak{A}_9$
all have weight $w=3$, their defect groups are conjugate to 
$Q_6\leq\mathfrak{A}_6$ and, in each case, up to conjugacy, there are four 
self-centralizing Brauer pairs:
$$\begin{array}{|ll|r|r|r|r|}
\hline
Q \quad (n\in\{6,7,9\},\,\,\, w=3) & 
  & |Q| & |Z(Q)| & x & d \\
\hline
\hline
\langle(1,2)(3,4),(1,3)(2,4)\rangle & \cong V_4     & 4 & 4 & 2 & 1 \\
\langle(1,2)(3,4),(3,4)(5,6)\rangle & \cong V_4     & 4 & 4 & 3 & 2 \\
\langle(1,2)(3,4),(1,3,2,4)(5,6)\rangle & \cong C_4 & 4 & 4 & 3 & 2 \\
\langle(1,2)(5,6),(1,3)(2,4)\rangle & \cong D_8     & 8 & 2 & 3 & 2 \\
\hline
\end{array}$$

\medskip
(c)
The non-principal block of $\mathfrak{A}_{11}$ has weight $w=4$, 
its defect groups are conjugate to $Q_8\leq\mathfrak{A}_8$ 
and, up to conjugacy, there are thirty-three self-centralizing Brauer pairs.
We do not mention all of them, but just one concluding the list of
cases in Remark \ref{rem:self}:
$$\begin{array}{|ll|r|r|r|r|}
\hline
Q \quad (n=11,\,\,\, w=4) & 
  & |Q| & |Z(Q)| & x & d \\
\hline
\hline
... & & & & & \\
\langle(1,2)(3,4),(1,2)(5,6),(5,6)(7,8)\rangle & \cong C_2^3 & 8 & 8 
                                               & 4 & 2 \\
... & & & & & \\
\hline
\end{array}$$
\end{expl}

\smallskip

\begin{rem}\normalfont
With the help of the computer algebra system {\sf MAGMA} \cite{MAGMA},
and using the techniques described in \cite{DKZ},
it can be shown that all the $2$-subgroups listed above
actually occur as vertices of suitable simple modules,
with the exception of the cyclic group in (b), of course,
and the group given in (c).

\smallskip

We also point out a mistake in \cite[Cor. 6.3(iii)]{DK}, where the 
$2$-groups $Q\leq \mathfrak{S}_n$
of order 4 occurring as vertices of simple $F\mathfrak{S}_n$-modules
were classified (up to $\mathfrak{S}_n$-conjugation).
In fact, the case where $Q=P_2\times P_2$ and $w=2$
cannot occur.
\end{rem}

\section{The Double Covers of $\mathfrak{S}_n$ and $\mathfrak{A}_n$}
\label{sec:schur} 

We begin by recalling the group presentations of the 
double covers of the symmetric and alternating groups,
as well as the necessary facts about their blocks.
Then we immediately proceed to prove the bound given
in Theorem \ref{thm:vtxbnd}.

\begin{noth}\label{noth:schurcov}\normalfont
{\sl Notation.}\;
(a)
Let $n\geq 1$, and consider the group 
$\widetilde{\mathfrak{S}}_n:=\langle z,t_1,\ldots,t_{n-1}\rangle$
with relations
\begin{align*}
z^2&=1,\\
zt_i&=t_iz,\text{ for } i=1,\ldots, n-1,\\
t_i^2&=z,\text{ for } i=1,\ldots, n-1, \quad(\ast)\\
t_it_j&=zt_jt_i, \text{ for } |i-j|>1,\\
(t_it_{i+1})^3&=z,\text{ for } i=1,\ldots, n-2; \quad(\ast\ast)
\end{align*}
in particular, we have $\widetilde{\mathfrak{S}}_1:=\langle z\rangle\cong C_2$.
Note that also $\widetilde{\mathfrak{S}}_n\leq\widetilde{\mathfrak{S}}_{n+1}$,
for $n\geq 1$. Via
$\theta:\widetilde{\mathfrak{S}}_n\longrightarrow \mathfrak{S}_n,\; 
 t_i\longmapsto (i,i+1)$,
we obtain a group epimorphism with central kernel $\langle z\rangle$. 

\medskip
Replacing the relations $(\ast)$ by $t_i^2=1$, for $i=1,\ldots, n-1$,
and the relations $(\ast\ast)$ by $(t_it_{i+1})^3=1$, for $i=1,\ldots, n-2$,
we get an isoclinic group $\widehat{\mathfrak{S}}_n$, which also
is a central extension of $\mathfrak{S}_n$;
we have $\widetilde{\mathfrak{S}}_n\not\cong\widehat{\mathfrak{S}}_n$
if and only if $1\neq n\neq 6$.
In the case where $n\geq 4$, the groups $\widetilde{\mathfrak{S}}_n$ 
and $\widehat{\mathfrak{S}}_n$ are the Schur 
representation groups of the symmetric group $\mathfrak{S}_n$.
Whenever we have a subgroup $H$ of $\mathfrak{S}_n$, 
we denote its preimage under $\theta$ by $\widetilde{H}$,
and similarly its preimage in $\widehat{\mathfrak{S}}_n$ is denoted
by $\widehat{H}$.

\medskip
In particular, for $H=\mathfrak{A}_n$, 
we get $\widetilde{\mathfrak{A}}_n\unlhd \widetilde{\mathfrak{S}}_n$
and $\widehat{\mathfrak{A}}_n\unlhd\widehat{\mathfrak{S}}_n$,
where we actually always have
$\widetilde{\mathfrak{A}}_n\cong\widehat{\mathfrak{A}}_n$, and
$|\widetilde{\mathfrak{S}}_n:\widetilde{\mathfrak{A}}_n|
=|\widehat{\mathfrak{S}}_n:\widehat{\mathfrak{A}}_n|=2$,
for $n\geq 2$.
If $n\geq 4$ and $6\neq n\neq 7$ then
$\widetilde{\mathfrak{A}}_n$ is the 
universal covering group of the alternating group $\mathfrak{A}_n$.
Since we have no distinction between $\widetilde{\mathfrak{A}}_n$
and $\widehat{\mathfrak{A}}_n$ anyway, and since it will turn out
that all observations for $\widetilde{\mathfrak{S}}_n$ immediately 
translate to $\widehat{\mathfrak{S}}_n$, we from now on confine
ourselves to investigating 
$\widetilde{\mathfrak{A}}_n\unlhd \widetilde{\mathfrak{S}}_n$.

\medskip
(b)
We list the known facts concerning the block theory of 
$\widetilde{\mathfrak{A}}_n$ and $\widetilde{\mathfrak{S}}_n$
we will need, 
where we from now on suppose that $p\geq 3$, 
for the remainder of this section.
Each faithful block $B$ of $F\widetilde{\mathfrak{S}}_n$ 
can be labelled combinatorially by some integer $w\geq 0$,
called the \textit{$p$-bar weight} of $B$, and a $2$-regular 
partition $\kappa$ of $n-pw$, called the \textit{$p$-bar core of $B$};
for details we refer to \cite[Appendix 10]{HH}.
Given the $p$-bar weight $w$ of the block $B$, by \cite[Thm. (1.3)]{O1}, 
the defect groups of $B$ are the
$\widetilde{\mathfrak{S}}_n$-conjugates of the Sylow $p$-subgroups 
of $\widetilde{\mathfrak{S}}_{pw}$.
The latter in turn are via $\theta$ mapped to Sylow $p$-subgroups 
of $\mathfrak{S}_{pw}$.
\end{noth}

\medskip
Arguing along the lines of \cite{DK}, we now have:

\begin{prop}\label{prop:boundcover}
Let $p\geq 3$, and let $B$ be a faithful block of 
$F\widetilde{\mathfrak{S}}_n$ with defect group $P$.
Let further $(Q,b_Q)$ be a self-centralizing Brauer $B$-pair. 
Then we have 
$$|P|\leq |Q|!.$$
\end{prop}

\begin{proof}
Let $w$ be the $p$-bar weight of $B$. Then $P$ is conjugate 
to a Sylow $p$-subgroup of $\widetilde{\mathfrak{S}}_{pw}$, 
and $\theta(P)$ is in $\mathfrak{S}_n$ conjugate to a Sylow 
$p$-subgroup of $\mathfrak{S}_{pw}$. 
Thus we have $\theta(P)=_{\mathfrak{S}_n} P_{pw}$, 
so that there is a subset $\Omega$ of $\{1,\ldots,n\}$
with $|\Omega|=pw$ and such that $\theta(P)$ acts fixed point freely
on $\Omega$ and fixes $\{1,\ldots,n\}\smallsetminus\Omega$ pointwise.
Moreover, by \cite[Thm. 5.5.21]{NT}, 
we may assume that
\begin{equation}\label{equ:Zcover}
Z(P)\leq C_{P}(Q)\leq Q\leq P.
\end{equation}
In particular, since $\theta|_P$ is injective, we infer that
$Z(\theta(P))=\theta(Z(P))\cong Z(P)$ 
acts fixed point freely on $\Omega$ as well. By (\ref{equ:Zcover}), we have
$Z(\theta(P))\leq \theta(Q)\leq \theta(P)$,
hence $\theta(Q)$ acts fixed point freely on $\Omega$ and fixes 
$\{1,\ldots,n\}\smallsetminus\Omega$ pointwise.
That is, $\theta(Q)\leq \mathfrak{S}(\Omega)$ 
and $Q\leq \widetilde{\mathfrak{S}}(\Omega)$. By \cite[Prop. 3.8 e]{Cab}, 
we further have
$$C_{\widetilde{\mathfrak{S}}(\Omega)}(Q)=
  Z(Q)\times Z(\widetilde{\mathfrak{S}}(\Omega))=
  Z(Q)\times\langle z\rangle.$$
Since $Q$ is a $p$-group, this implies
$$C_{\mathfrak{S}(\Omega)}(\theta(Q))=
  \theta(C_{\widetilde{\mathfrak{S}}(\Omega)}(Q))=\theta(Z(Q))= 
  Z(\theta(Q)),$$
thus, applying \cite[Thm. 5.1]{DK}, we get
$|P|=|\theta(P)|\leq|\mathfrak{S}(\Omega)|=|\Omega|!\leq |\theta(Q)|!=|Q|!$.
\end{proof}

\section{The Weyl Groups $\mathfrak{B}_n$ and $\mathfrak{D}_n$}
\label{sec:weyl}

We begin by recalling the description of the Weyl groups
$\mathfrak{B}_n$ and $\mathfrak{D}_n$ of type $B_n$ and $D_n$, respectively.
Then we recall the necessary facts about their blocks,
distinguishing the cases $p$ even and $p$ odd,
in order to immediately proceed to prove the bound given
in Theorem \ref{thm:vtxbnd}.

\begin{noth}\label{noth:setting}
\normalfont
{\sl Notation.}\;
Let $C_2:=\langle (1,2)\rangle$ be the group of order 2,
and for $n\in\mathbb{N}$ 
set 
$$\mathfrak{B}_n:=C_2\wr\mathfrak{S}_n=\{(x_1,\ldots,x_n;\pi)\mid 
  x_1,\ldots , x_n\in C_2,\; \pi\in\mathfrak{S}_n\}.$$
We define 
$\mathfrak{S}_n^*:=\{(1,\ldots,1;\pi)\mid 
\pi\in\mathfrak{S}_n\}\leq \mathfrak{B}_n$, 
and denote the usual group isomorphism 
$\mathfrak{S}_n^*\longrightarrow \mathfrak{S}_n$ by $\varphi$. 
Moreover, whenever $U$ is a subgroup of $\mathfrak{S}_n$, we denote
by $U^*\leq \mathfrak{S}_n^*$ its image under $\varphi^{-1}$.
Let also 
$H:=\{(x_1,\ldots,x_n;1)\mid x_1,\ldots , x_n\in C_2\}\unlhd\mathfrak{B}_n$ 
be the base group of 
$\mathfrak{B}_n$; so $H$ is isomorphic to a direct
product of $n$ copies of $C_2$, and we have
$\mathfrak{B}_n=H\mathfrak{S}_n^*$.
We will identify $\mathfrak{B}_n$ with a subgroup of 
$\mathfrak{S}_{2n}$, in the usual way by the primitive action.
Furthermore, let $\mathfrak{D}_n:=\mathfrak{B}_n\cap \mathfrak{A}_{2n}$.
For $n\geq 2$, the group $\mathfrak{B}_n$ is isomorphic to the 
\textit{Weyl group of type $B_n$}, and for $n\geq 4$, the group
$\mathfrak{D}_n$ is isomorphic 
to the \textit{Weyl group of type $D_n$}; see \cite[4.1.33]{JK}. 
\end{noth}

\begin{rem}\normalfont
As has been pointed out by the referee, Feit's Conjecture for
the Weyl groups $\mathfrak{B}_n$ and $\mathfrak{D}_n$
can also be deduced directly from a more general theorem
on semidirect products with abelian kernel;
we will state and prove this theorem in Section \ref{sec:semidirect}
below. 
We will, however, treat the groups
$\mathfrak{B}_n$ and $\mathfrak{D}_n$
separately, 
in order to stick to our general strategy 
for proving Feit's Conjecture by relating it to Puig's Conjecture via
Theorem \ref{thm:feitred}.
\end{rem}

\begin{noth}\label{noth:weyleven}
\normalfont
{\sl The case $p=2$.}\;
(a)
Suppose first that $p=2$.
Since $H$ is a normal $2$-subgroup of $\mathfrak{B}_n$ 
such that $C_{\mathfrak{B}_n}(H)=H$ is a $2$-group,
it follows from \cite[Thm. 5.2.8]{NT} (see also \cite[Exc. 5.2.10]{NT}) that
$F\mathfrak{B}_n$ has only the principal block. 
Moreover, $H$ acts trivially on every simple $F\mathfrak{B}_n$-module. 
Thus if $Q$ is a vertex of a simple $F\mathfrak{B}_n$-module
then $H\leq Q$. In particular, we have $|Q|\geq 2^n$, hence
$n\leq \log_2(|Q|)$. Therefore, if $P$ is a Sylow $2$-subgroup
of $\mathfrak{B}_n$ then
$$|P|\leq |\mathfrak{B}_n|=2^n\cdot n!\leq |Q|\cdot\log_2(|Q|)!.$$

\medskip
(b)
To deal with $\mathfrak{D}_n$, note that 
$\mathfrak{D}_n=\mathfrak{B}_n\cap\mathfrak{A}_{2n}= 
 (H\cap\mathfrak{A}_{2n})\mathfrak{S}_{n}^*$.
An argument analogous to the one used in Part (a) above
shows that also $F\mathfrak{D}_n$ has only
the principal block and, whenever $Q$ is a vertex of a simple 
$F\mathfrak{D}_n$-module, $H\cap\mathfrak{A}_{2n}\leq Q$.
In particular, we have $|Q|\geq 2^{n-1}$, hence $n\leq\log_2(|Q|)+1$. 
Therefore, if $P$ is a Sylow $2$-subgroup of $\mathfrak{D}_n$
then we deduce
$$|P|\leq |\mathfrak{D}_n|=2^{n-1}\cdot n!\leq |Q|\cdot(\log_2(|Q|)+1)!.$$
\end{noth}

\begin{noth}\label{noth:weylblocks}
\normalfont
{\sl The case $p\geq 3$.}\;
Let now $p\geq 3$, for the remainder of this section.
We briefly recall the well-known structure of the defect groups 
of the blocks of $F\mathfrak{B}_n$. 
By the Theorem of Fong--Reynolds \cite[Thm. 5.5.10]{NT}
applied to the base group $C_2^n\cong H\unlhd\mathfrak{B}_n$,
see also \cite{Os2} or \cite[Ch. 4]{JK}, 
the blocks of $F\mathfrak{B}_n$ are parametrized by pairs
$(\kappa,w)$, with $\kappa=(\kappa_0,\kappa_1)$ and $w=(w_0,w_1)$,
and where, for $i=0,1$, the partition $\kappa_i$ is the $p$-core 
of some partition of $n_i:=|\kappa_i|+p w_i$ such that $n=n_0+n_1$.
Moreover, the inertial group of the block $B(\kappa,w)$ is given as 
$T_{\mathfrak{B}_n}(B(\kappa,w)):=
 (C_2\wr\mathfrak{S}_{n_0})\times(C_2\wr\mathfrak{S}_{n_1})
=C_2\wr (\mathfrak{S}_{n_0}\times\mathfrak{S}_{n_1})
 \leq\mathfrak{B}_n$.

\medskip
Note that 
every $p$-subgroup of $\mathfrak{B}_n$ is  
conjugate to a subgroup of $\mathfrak{S}_n^*$. 
If $P$ is a defect group of the block 
$B(\kappa,w)$ then $P$ is conjugate to a Sylow $p$-subgroup 
$P_{pw_0}^*\times P_{pw_1}^*$ of 
$\mathfrak{S}_{pw_0}^*\times\mathfrak{S}_{pw_1}^*\leq
 \mathfrak{S}_{n_0}^*\times\mathfrak{S}_{n_1}^*\leq\mathfrak{S}_n^*$.
Note that 
$\varphi(P_{pw_i}^*)=P_{pw_i}\leq \mathfrak{S}_{pw_i}\leq\mathfrak{S}_{n_i}$,
for $i=0,1$, is a defect group of the block of $\mathfrak{S}_{n_i}$ 
parametrized by $\kappa_i$.
\end{noth}

\begin{prop}\label{prop:puig_weylB}
Let $p\geq 3$, let $D$ be a simple $F\mathfrak{B}_n$-module belonging 
to a block with defect group $P$, and let $Q$ be a vertex of $D$. 
Then we have
$$ |P|\leq |Q|! .$$ 
\end{prop}

\begin{proof}
Let $B(\kappa,w)$ be the block in question,
and $n_i:=|\kappa_i|+p w_i$, for $i=0,1$.
Hence we may assume that
$Q\leq P=P_{pw_0}^*\times P_{pw_1}^*\leq
 \mathfrak{S}_{n_0}^*\times\mathfrak{S}_{n_1}^*\leq\mathfrak{S}_n^*$.
Let 
$T:=T_{\mathfrak{B}_n}(B(\kappa,w))
 =(C_2\wr\mathfrak{S}_{n_0})\times(C_2\wr\mathfrak{S}_{n_1})
 \leq\mathfrak{B}_n$
be the inertial group associated with $B(\kappa,w)$.
Then the simple $F\mathfrak{B}_n$-modules belonging to $B(\kappa,w)$
can be described as follows (see \cite[Sec. 4.3]{JK}, and \cite{Os2} ):

\medskip
Let $F$ be the trivial $FC_2$-module, and let $E$ be the non-trivial
simple $FC_2$-module. Then the outer tensor product $F^{\otimes n_0}\otimes_F E^{\otimes n_1}$
naturally becomes an $FT$-module.
Letting $D_i$, for $i=0,1$, be a
simple $F\mathfrak{S}_{n_i}$-module
in the block parametrized by $\kappa_i$,
the tensor product $D_0 \otimes_F D_1$ becomes an 
$F[\mathfrak{S}_{n_0}\times\mathfrak{S}_{n_1}]$-module with respect
to the outer-tensor-product action. Inflating with respect to 
the base group $C_2^{n_0}\times C_2^{n_1}\unlhd T$ yields 
the simple $FT$-module
$\Inf_{C_2^{n_0}\times C_2^{n_1}}^{T}(D_0 \otimes_F D_1)$.
Then inducing the ordinary tensor product
$$ M:=(F^{\otimes n_0}\otimes_F E^{\otimes n_1})\otimes_F
     \Inf_{C_2^{n_0}\times C_2^{n_1}}^{T}(D_0 \otimes_F D_1) ,$$
which is a simple $FT$-module, to $\mathfrak{B}_n$ we 
get a simple $F\mathfrak{B}_n$-module
$\ind_{T}^{\mathfrak{B}_n}(M)$ belonging to $B(\kappa,w)$. 
Conversely, every simple
$F\mathfrak{B}_n$-module belonging to the block
$B(\kappa,w)$ arises in such a way. 

\medskip
Hence,
$D\cong \ind_{T}^{\mathfrak{B}_n}(M)$,
for suitably chosen $D_0$ and $D_1$.
Let $Q_i\leq\mathfrak{S}_{n_i}$ be a vertex of $D_i$, for $i=0,1$,
where we may assume that $Q_i\leq P_{pw_i}$.
Hence $Q_0\times Q_1\leq\mathfrak{S}_{n_0}\times\mathfrak{S}_{n_1}$
is a vertex of the outer tensor product $D_0 \otimes_F D_1$.
Since $F$ and $E$ are projective $FC_2$-modules, 
letting $Q_i^*:=\varphi^{-1}(Q_i)$, it follows from \cite{Ku} that 
$Q_0^*\times Q_1^*\leq\mathfrak{S}_{n_0}^*\times\mathfrak{S}_{n_1}^*\leq 
 \mathfrak{S}_{n}^*\leq\mathfrak{B}_n$
is a vertex of $\ind_{T}^{\mathfrak{B}_n}(M)\cong D$.

\medskip
By \cite{DK} we have $pw_i\leq |Q_i|$, and hence 
$$ |P|=|P_{pw_0}^*|\cdot|P_{pw_1}^*|
  \leq (pw_0)!\cdot(pw_1)!\leq |Q_0|!\cdot |Q_1|!
  \leq (|Q_0|\cdot |Q_1|)!=|Q|! .
$$
\end{proof}

\section{Semidirect Products}\label{sec:semidirect}

As was pointed out by the referee, Feit's conjecture 
can be proven for general semidirect products with 
abelian kernel, provided it holds for the complements
occurring and all their subgroups. We proceed to state
and prove this.

\begin{noth}\label{noth:semidirect}
\normalfont
{\sl Simple modules of semidirect products.}\, Suppose that
$G$ is a semidirect product $G=H\rtimes_\alpha U$ of
an abelian group $H$ with a group $U$, with respect to 
a group homomorphism $\alpha:U\longrightarrow\Aut(H)$.
We recall the well-known construction of the simple $FG$-modules from
those of subgroups of $G$, which is a consequence of Clifford's Theorem;
for a proof we refer to \cite[Thm. 11.1, Thm. 11.20, Exc. 11.13]{CRI}.

\medskip
Suppose that $E$ is a simple $FH$-module, which is, in particular,
one-dimensional, since $H$ is abelian. 
Let further $T_G(E)$ be the inertial group of $E$ in $G$; thus
$T_G(E)=H\rtimes_\alpha (U\cap T_G(E))$,
where we denote the restriction of $\alpha$ to $U\cap T_G(E)$ by
$\alpha$ again.
Then we can extend $E$ to a simple $FT_G(E)$-module, 
which we denote by $E$ again, by letting
$$(hu)\cdot x:= hx\quad (h\in H, u\in U\cap T_G(E),\, x\in E) ;$$
to see that this indeed yields an $FT_G(E)$-module 
just note that we have
$$ {}^u h\cdot x = h\cdot x \quad (h\in H, u\in U\cap T_G(E),\, x\in E). $$

Set $T_{U,\alpha}(E):=U\cap T_G(E)$, hence we have
$T_G(E)/H\cong T_{U,\alpha}(E)$,
and let $E'$ be a simple $FT_{U,\alpha}(E)$-module. Then the inflation
$\Inf_H^{T_G(E)}(E')$ is a simple $FT_G(E)$-module, as is the tensor
product $E\otimes_F\Inf_H^{T_G(E)}(E')$. Moreover, the induction
$D(E,E'):=\ind_{T_G(E)}^G(E\otimes_F\Inf_H^{T_G(E)}(E'))$ 
is a simple $FG$-module.

\medskip
Now, let $\mathcal{E}$ be
a transversal for the isomorphism classes of simple
$FH$-modules.
Then, as $E$ varies over $\mathcal{E}$,
and $E'$ varies over 
a transversal for the isomorphism classes of
simple $FT_{U,\alpha}(E)$-modules, 
$D(E,E')$ varies over a transversal for the isomorphism classes of simple
$FG$-modules.
\end{noth}

With the above notation, we have the following
\begin{thm}\label{thm:semidirect}
Let $H$ be an abelian group, let $\mathcal{U}$ be a set
of groups, 
and let
$$ H\rtimes\mathcal{U}:=\{H\rtimes_\alpha U\mid  
   U\in\mathcal{U},\, \alpha:U\longrightarrow\Aut(H) \} .$$
Suppose that Feit's Conjecture holds for
$$ \mathcal{T}(H\rtimes\mathcal{U}):=\{T_{U,\alpha}(E)\mid
   U\in\mathcal{U},\, \alpha:U\longrightarrow\Aut(H),\, E\in\mathcal{E}\} .$$
Then Feit's Conjecture holds for 
$H\rtimes\mathcal{U}$ as well.
\end{thm}

\begin{proof}
Let $G=H\rtimes_\alpha U$ be a group in 
$H\rtimes\mathcal{U}$,
and let $D$ be a simple $FG$-module.
As we have just seen in \ref{noth:semidirect}, 
there are a simple $FH$-module $E$ and
a simple $FT_{U,\alpha}(E)$-module $E'$ such that
$$D\cong\ind_{T_G(E)}^G(E\otimes_F\Inf_H^{T_G(E)}(E')).$$

Since $\dim_F(E)=1$, it is a trivial-source module,
and moreover tensoring with $E$ is a vertex- and source-preserving
auto-equivalence of the module category of $T_G(E)$.
Hence every vertex of $\Inf_H^{T_G(E)}(E')$ is a also vertex of
$E\otimes_F\Inf_H^{T_G(E)}(E')$, and every source of 
$\Inf_H^{T_G(E)}(E')$ is also a source of
$E\otimes_F\Inf_H^{T_G(E)}(E')$.
Moreover, $\ind_{T_G(E)}^G(E\otimes_F\Inf_H^{T_G(E)}(E'))$ 
has some indecomposable direct summand that has 
a vertex and an associated source in common
with $E\otimes_F\Inf_H^{T_G(E)}(E')$. Thus, since $D$ 
is, in particular, indecomposable, 
the vertex-source pairs of $\Inf_H^{T_G(E)}(E')$
and those of $D$ coincide.

\medskip
So, suppose that $Q$ is a vertex of $\Inf_H^{T_G(E)}(E')$,
and $L$ is a $Q$-source.
Then, by \cite[Prop. 2.1]{Ku} and \cite[Prop. 2]{Harr}, 
we deduce that $QH/H$ is a vertex of $E'$ and
that $L= \res_Q^{QH}(\Inf_H^{QH}(\bar{L}))$, for some
$QH/H$-source $\bar{L}$ of $E'$.

\medskip
Consequently, given any $p$-group $Q$, the above arguments imply
$$|\mathcal{V}_{H\rtimes\mathcal{U}}(Q)|\leq 
  |\mathcal{V}_{\mathcal{T}(H\rtimes\mathcal{U})}(QH/H)|.$$
Since the latter cardinality is finite by our hypothesis, the 
assertion of the theorem follows.
\end{proof}


\medskip
{\sc S.D.:
Mathematical Institute, University of Oxford \\
24-29 St Giles', Oxford, OX1 3LB, UK} \\
{\sf danz@maths.ox.ac.uk}

\medskip
{\sc J.M.:
Lehrstuhl D f\"ur Mathematik, RWTH Aachen \\
Templergraben 64, D-52062 Aachen, Germany} \\
{\sf Juergen.Mueller@math.rwth-aachen.de}

\end{document}